\documentclass{article}
\usepackage[fleqn,reqno]{amsmath}

\newtheorem{theorem}{Theorem}
\newtheorem{proposition}{Proposition}
\newtheorem{lemma}{Lemma}
\newtheorem{remark}{Remark}
\numberwithin{equation}{section}
\numberwithin{theorem}{section}
\numberwithin{lemma}{section}
\numberwithin{proposition}{section}
\numberwithin{remark}{section}
\newenvironment{proof}[1][Proof]{\noindent\textbf{#1.} }{\ \rule{0.5em}{0.5em}}
\input{tcilatex}
\input{tcilatex}

\begin{document}

\title{Dynamics of ellipses inscribed in triangles}
\author{Alan Horwitz \\
Professor Emeritus of Mathematics\\
Penn State Brandywine\\
25 Yearsley Mill Rd.\\
Media, PA 19063\\
alh4@psu.edu}
\date{6/19/15}
\maketitle

\begin{abstract}
Suppose that we are given two distinct points, $P_{1}$ and $P_{2}$, in the
interior of a triangle, $T$. Is there always an ellipse inscribed in $T$
which also passes through $P_{1}$ and $P_{2}$ ? If yes, how many such
ellipses ? We answer those questions in this paper. It turns out that,
except for $P_{1}$ and $P_{2}$ on a union of three line segments, there are
four such ellipses which pass through $P_{1}$ and $P_{2}$. We also answer a
similar question if instead $P_{1}$ and $P_{2}$ lie on the boundary of $T$.
Finally, an interesting related question, is the following: Given a point, $%
P $, in the interior of a triangle, $T$, and a real number, $r$, is there
always an ellipse inscribed in $T$ which passes through $P$ and has slope $r$
at $P$ ? Again, if yes, how many such ellipses ? The answer is somewhat
different than for the two point case without specifying a slope. There are
cases where no such ellipse exists.
\end{abstract}

\bigskip {\LARGE Introduction}

It is not hard to see that if we are given one point, $P$, in the interior
of a triangle, $T$, then there are infinitely many ellipses inscribed in $T$
and which also pass through $P$. By inscribed in $T$ we mean that the
ellipse lies in $T$ and is tangent to each side of $T$. Also, if we are
given three points in the interior of $T$, then it also follows easily that
there might not be any ellipse inscribed in $T$ and which also passes
through all three of those points. So we decided to look at the average for
the number of points: That is, suppose that we are given two distinct
points, $P_{1}$ and $P_{2}$, in the interior of $T$. Is there always an
ellipse inscribed in $T$ which also passes through $P_{1}$ and $P_{2}$ ? If
yes, how many such ellipses ? We answer those questions below(see Theorem %
\ref{TwoPoints}). It turns out that, except for $P_{1}$ and $P_{2}$ on a
union of three line segments, there are four such ellipses which pass
through $P_{1}$ and $P_{2}$. We encourage the reader to try to determine
what those line segments are and how many ellipses pass through $P_{1}$ and $%
P_{2}$ if they lie on one of those segments. We also answer a similar
question if instead $P_{1}$ and $P_{2}$ lie on the boundary of $T$(see
Theorem \ref{T3}).

Finally, an interesting related question, is the following: Given a point, $%
P $, in the interior of a triangle, $T$, and a real number, $r$, is there
always an ellipse inscribed in $T$ which passes through $P$ and has slope $r$
at $P$ ? Again, if yes, how many such ellipses ? The answer here(see Theorem %
\ref{OnePoint}) is somewhat different than for the two point case without
specifying a slope. There are cases where no such ellipse exists. Again, we
encourage the reader to try to determine what those cases are.

The reason that we used the word dynamics in the title of this article is
the following: Imagine a particle constrained to travel along the path of an
ellipse inscribed in a triangle, $T$. Thus the particle bounces off each
side of $T$\ along its path. Of course there are infinitely many such paths.
Can we also specify two points in $T$, or a point in $T$ and a slope at that
point,\ which the particle must pass through ? If yes, is such a path then
unique ?

At the end we provide some algorithms for finding the ellipses described
above, when they exist. A key to our methods is the general equation of an
ellipse inscribed in the unit triangle(see Proposition \ref{P1}).

\section{Two Points--Interior}

\begin{theorem}
\label{TwoPoints}Let $P_{1}=(x_{1},y_{1})$ and $P_{2}=(x_{2},y_{2})$ be
distinct points which lie in $\limfunc{int}\left( T\right) =$ interior of
the triangle, $T$, with vertices $A,B,C$.

(i) Suppose that $P_{1}$ and $P_{2}$ do \textbf{not }lie on the same line
thru any of the vertices of $T$. Then there are precisely four distinct
ellipses inscribed in $T$ which pass through $P_{1}$ and $P_{2}$.

(ii) Suppose that\ $P_{1}$ and $P_{2}$ do lie on the same line thru one of
the vertices of $T$. Then there are precisely two distinct ellipses
inscribed in $T$ which pass through $P_{1}$ and $P_{2}$.
\end{theorem}

By affine invariance, it suffices to prove Theorem \ref{TwoPoints} for the 
\textit{unit} triangle, $T$--the triangle with vertices $(0,0),(1,0)$, and $%
(0,1)$. If $P_{1}=(x_{1},y_{1})$ and $P_{2}=(x_{2},y_{2})$ are distinct
points which lie in $\limfunc{int}\left( T\right) =$ \textit{interior} of
the triangle, $T$, then it follows that

\begin{eqnarray}
0 &<&x_{1}\neq x_{2}<1  \notag \\
0 &<&y_{1}\neq y_{2}<1  \label{1} \\
x_{1}+y_{1} &<&1,x_{2}+y_{2}<1\text{.}  \notag
\end{eqnarray}

The theorem above then takes the following form:

\begin{theorem}
\label{T1}Let $P_{1}=(x_{1},y_{1})$ and $P_{2}=(x_{2},y_{2})$ be distinct
points which lie in the interior of the unit triangle.

(i) Suppose that $x_{2}y_{1}-x_{1}y_{2}\neq
0,(1-x_{2})y_{1}-(1-x_{1})y_{2}\neq 0$, and $x_{2}(1-y_{1})-x_{1}(1-y_{2})%
\neq 0$. Then there are precisely four distinct ellipses inscribed in $T$
which pass through $P_{1}$ and $P_{2}$.

(ii) Suppose that $y_{1}x_{2}-x_{1}y_{2}=0$,\textbf{\ }$%
(1-x_{2})y_{1}-(1-x_{1})y_{2}=0$, or $x_{2}(1-y_{1})-x_{1}(1-y_{2})=0$. Then
there are precisely two distinct ellipses inscribed in $T$ which pass
through $P_{1}$ and $P_{2}$.
\end{theorem}

\subsection{Preliminary Results}

Before proving Theorem \ref{T1}, we need several preliminary results.
Throughout this section we assume that $T$ is the unit triangle --the
triangle with vertices $(0,0),(1,0)$, and $(0,1)$, and that $%
(x_{1},y_{1}),(x_{2},y_{2})\in \limfunc{int}\left( T\right) $. We also
assume throughout that, unless stated otherwise, (\ref{1}) holds . Let 
\begin{equation}
G=\left\{ (x,y)\in \Re ^{2}:0<x<1,0<y<1,x+y<1\right\} \text{,}  \label{G}
\end{equation}%
so that $P_{1},P_{2}\in G$.

\begin{lemma}
\label{L1}Only one of the following three conditions can hold for any given
distinct points $P_{1}=(x_{1},y_{1})$ and $P_{2}=(x_{2},y_{2})$ which lie in
the interior of $T$.

(1) $x_{1}y_{2}-y_{1}x_{2}=0$, (2) $(1-x_{2})y_{1}-(1-x_{1})y_{2}=0$, (3) $%
x_{2}(1-y_{1})-x_{1}(1-y_{2})=0$
\end{lemma}

\begin{proof}
Clearly distinct points $P_{1}$ and $P_{2}$ which lie in $\limfunc{int}%
\left( T\right) $ cannot lie on the same line through $(0,0)$ and $(1,0)$,
on the same line through $(0,0)$ and $(0,1)$, or on the same line through $%
(1,0)$ and $(0,1)$.
\end{proof}

\begin{lemma}
\label{L12}(i) $x_{1}y_{2}+x_{2}y_{1}+2y_{1}y_{2}-y_{1}-y_{2}<0$ and $%
x_{1}y_{2}+x_{2}y_{1}+2x_{1}x_{2}-x_{2}-x_{1}<0$.

(ii) $2x_{1}x_{2}-x_{1}-x_{2}<0$ and $2y_{1}y_{2}-y_{1}-y_{2}<0$.
\end{lemma}

\begin{proof}
Note that $\tfrac{1-x_{1}}{y_{1}}>1$ and $\tfrac{1-x_{2}}{y_{2}}>1$. Then $%
x_{1}y_{2}+x_{2}y_{1}+2y_{1}y_{2}-y_{1}-y_{2}=y_{1}y_{2}\left( 2-\tfrac{%
1-x_{1}}{y_{1}}-\tfrac{1-x_{2}}{y_{2}}\right) <0$. Similarly, $\tfrac{1-y_{1}%
}{x_{1}}>1$\ and $\tfrac{1-y_{2}}{x_{2}}>1$, which implies that $%
x_{1}y_{2}+x_{2}y_{1}+2x_{1}x_{2}-x_{2}-x_{1}=x_{1}x_{2}\left( 2-\tfrac{%
1-y_{1}}{x_{1}}-\tfrac{1-y_{2}}{x_{2}}\right) <0$, which proves (i). (ii)
then follows immediately from (i).
\end{proof}

\begin{lemma}
\label{L2}Let $V$ be the interior of the medial triangle of $T$ $=$ triangle
with vertices at the midpoints of the sides of $T$, and let $S$ be the unit
square $=(0,1)\times (0,1)$. Then

$V=\left\{ \left( \dfrac{1}{2}\dfrac{t}{w+(1-w)t},\dfrac{1}{2}\dfrac{w}{%
w+(1-w)t}\right) \right\} _{(w,t)\in S}$.
\end{lemma}

\begin{proof}
It follows easily that $(x,y)\in V\iff \tfrac{1}{2}-x<y<\tfrac{1}{2}\ $and $%
0<x<\tfrac{1}{2}$. Now suppose that $(w,t)\in S$, and let $x=\tfrac{1}{2}%
\tfrac{t}{w+(1-w)t},y=\tfrac{1}{2}\tfrac{w}{w+(1-w)t}$.

$\tfrac{t}{w+(1-w)t}-1=-\allowbreak \tfrac{w(1-t)}{w+(1-w)t}<0$, which
implies that $\tfrac{t}{w+(1-w)t}<1$ and so $0<x<\tfrac{1}{2}$; $\tfrac{w}{%
w+(1-w)t}-1=\allowbreak -\tfrac{(1-w)t}{w+(1-w)t}<0$, which implies that $%
\tfrac{w}{w+(1-w)t}<1$ and so $y<\tfrac{1}{2}$; Finally, $\tfrac{w+t}{%
w+(1-w)t}-1=\allowbreak \tfrac{wt}{w+(1-w)t}>0$, which implies that $\tfrac{%
w+t}{w+(1-w)t}>1$. Thus $x+y=\tfrac{1}{2}\tfrac{t}{w+(1-w)t}+\tfrac{1}{2}%
\tfrac{w}{w+(1-w)t}=\tfrac{1}{2}\tfrac{w+t}{w+(1-w)t}>\tfrac{1}{2}$, and
hence $\left( \tfrac{1}{2}\tfrac{t}{w+(1-w)t},\tfrac{1}{2}\tfrac{w}{w+(1-w)t}%
\right) \in V$. Conversely, suppose that $\left( \tfrac{1}{2}\tfrac{t}{%
w+(1-w)t},\tfrac{1}{2}\tfrac{w}{w+(1-w)t}\right) \in V$. The unique solution
of the system of equations $x=\tfrac{1}{2}\tfrac{t}{w+(1-w)t},y=\tfrac{1}{2}%
\tfrac{w}{w+(1-w)t}$ is $w=\tfrac{1}{2}\tfrac{2x+2y-1}{x},t=\tfrac{2xw}{%
2xw-2x+1}$; $x+y>\tfrac{1}{2},0<x<\tfrac{1}{2}$ implies that $w,t>0$; $%
\tfrac{1}{2}\tfrac{2x+2y-1}{x}-1=\allowbreak \tfrac{1}{2}\tfrac{2y-1}{x}<0$
since $y<\tfrac{1}{2}$, and thus $w<1$; $\tfrac{1}{2}\tfrac{2x+2y-1}{y}%
-1=\allowbreak \tfrac{1}{2}\tfrac{2x-1}{y}<0$ since $x<\tfrac{1}{2}$, and
thus $t<1$. Hence $(w,t)\in S$.
\end{proof}

\qquad We state the following known result without proof. The first
inequality insures that the conic is an ellipse, while the second insures
that the ellipse is non--trivial.

\begin{lemma}
\label{L3}The equation $Ax^{2}+By^{2}+2Cxy+Dx+Ey+F=0$, with $A,B>0$, is the
equation of an ellipse if and only if $AB-C^{2}>0$ and $%
AE^{2}+BD^{2}+4FC^{2}-2CDE-4ABF>0$.
\end{lemma}

\begin{proposition}
\label{P1}(i) $E$ is an ellipse inscribed in $T$\ if and only if the general
equation of $E$ is given by 
\begin{equation}
w^{2}x^{2}+t^{2}y^{2}-2wt\left( 2wt-2w-2t+1\right)
xy-2w^{2}tx-2t^{2}wy+t^{2}w^{2}=0  \label{2}
\end{equation}

for some $(w,t)\in S=(0,1)\times (0,1)$. Furthermore,

(ii) If $E$ is the ellipse given in (i) with equation (\ref{2}) for some $%
0<t<1,0<w<1$, then $E$ is tangent to the three sides of $T$\ at the points $%
T_{1}=(t,0),T_{2}=(0,w)$, and $T_{3}=\left( \dfrac{t(1-w)}{t+(1-2t)w},\dfrac{%
w(1-t)}{t+(1-2t)w}\right) $.
\end{proposition}

\begin{proof}
First, suppose that $E$ is given by (\ref{2}) for some $(w,t)\in S$. Then $E$
has the form $Ax^{2}+By^{2}+2Cxy+Dx+Ey+F=0$, where $A=w^{2}$, $B=t^{2}$, $%
C=-wt\left( 2wt-2w-2t+1\right) $, $D=-2w^{2}t$, $E=-2t^{2}w$, and $%
F=t^{2}w^{2}$. $AB-C^{2}=\allowbreak 4w^{2}t^{2}\left( 1-t\right) \left(
1-w\right) {\large (}(1-t)w+t{\large )}>0$ and $%
AE^{2}+BD^{2}+4FC^{2}-2CDE-4ABF=\allowbreak 16w^{4}t^{4}\left( 1-w\right)
^{2}\left( 1-t\right) ^{2}>0$. Thus by Lemma \ref{L3}, (\ref{2}) defines the
equation of an ellipse for any $(w,t)\in S$. Now let $%
F(x,y)=w^{2}x^{2}+t^{2}y^{2}-2wt\left( 2tw-2w-2t+1\right)
xy-2w^{2}tx-2t^{2}wy+t^{2}w^{2}$, the left hand side of (\ref{2}). Then $%
F(t,0)=\allowbreak 0$, $F(0,w)=\allowbreak 0$, and $F\left( \tfrac{t(1-w)}{%
t+(1-2t)w},\tfrac{w(1-t)}{t+(1-2t)w}\right) =\allowbreak 0$, which implies
that the three points $T_{1},T_{2}$, and $T_{3}$ lie on $E$. Differentiating
both sides of the equation in (\ref{2}) with respect to $x$ yields $\dfrac{dy%
}{dx}=D(x,y)$, where $D(x,y)=-\tfrac{w(2ywt^{2}-2ywt-2yt^{2}+wt+yt-xw)}{%
t\left( 2xw^{2}t-2xwt-2xw^{2}+wt+xw-yt\right) }$. $D(t,0)=\allowbreak 0=$
slope of horizontal side of $T$ and $D\left( \tfrac{t(1-w)}{t+(1-2t)w},%
\tfrac{w(1-t)}{t+(1-2t)w}\right) =\allowbreak -1=$ slope of the hypotenuse
of $T$. When $x=0,y=w$, the denominator of $D(x,y)$ equals $0$, but the
numerator of $D(x,y)$ equals $\allowbreak 2tw^{2}\left( 1-w\right) \left(
1-t\right) \neq 0$. Thus $E$ is tangent to the vertical side of $T$. For any
simple closed convex curve, such as an ellipse, tangent to each side of $T$
then implies that that curve lies in $T$. Since it follows easily that $%
T_{1},T_{2}$, and $T_{3}$ lie on the three sides of $T$, that proves that $E$
is inscribed in $T$. Second, suppose that $E$ is an ellipse inscribed in $T$%
. It is well known \cite{C} that each point of $V$, the medial triangle of $T
$, is the center of one and only one ellipse inscribed in $T$, and thus the
center of $E$ lies in $V$. By Lemma \ref{L2}, the center of $E$ has the form 
$\left( \tfrac{1}{2}\tfrac{t}{w+(1-w)t},\tfrac{1}{2}\tfrac{w}{w+(1-w)t}%
\right) ,(w,t)\in S$. Now it is not hard to show that each ellipse given by (%
\ref{2}) has center $\left( \tfrac{1}{2}\tfrac{t}{w+(1-w)t},\tfrac{1}{2}%
\tfrac{w}{w+(1-w)t}\right) $. We have just shown that (\ref{2}) represents a
family of ellipses inscribed in $T$\ as $(w,t)$ varies over $S$, so if $E$
were not given by (\ref{2}) for some $(w,t)\in S$, then there would have to
be two ellipses inscribed in $T$ and with the same center. That cannot
happen since each point of $V$ is the center of only one ellipse inscribed
in $T$. That proves (i). We have also just shown that if $E$ is given by (%
\ref{2}), then $E$ is tangent to the three sides of $T$\ at the points $%
T_{1},T_{2}$, and $T_{3}$, which proves (ii).
\end{proof}

We find it convenient to introduce the following notation for fixed $%
(x_{1},x_{2},y_{1},y_{2})$.

\begin{equation}
q_{i}(t)=\allowbreak (x_{i}-t)^{2}+4x_{i}y_{i}t(1-t),i=1,2\text{.}  \notag
\end{equation}

Note that $q_{i}$ has discriminant $4x_{i}^{2}\left( 1-2y_{i}\right)
^{2}-4\left( 1-4x_{i}y_{i}\right) x_{i}^{2}=\allowbreak $

$16x_{i}^{2}y_{i}\left( x_{i}+y_{i}-1\right) <0$ by (\ref{1}), which implies
that $q_{i}$ has no real roots, $i=1,2$. Since $q_{i}(0)>0$, it follows that 
\begin{equation}
q_{i}(t)>0,t\in \Re ,i=1,2\text{.}  \label{q}
\end{equation}

\begin{proposition}
\label{P2}Let $P_{1}=(x_{1},y_{1})$ and $P_{2}=(x_{2},y_{2})$ be distinct
points in $T$. Then there is an ellipse inscribed in $T$ which passes
through $P_{1}$ and $P_{2}$ if and only if the following system of equations
hold for some $(w,t)\in S=(0,1)\times (0,1)$:%
\begin{eqnarray}
q_{1}(t)w^{2}+2ty_{1}{\large (}(2x_{1}-1)t-x_{1}{\large )}\allowbreak
w+t^{2}y_{1}^{2} &=&0  \label{3} \\
q_{2}(t)w^{2}+2ty_{2}{\large (}(2x_{2}-1)t-x_{2}{\large )}\allowbreak
w+t^{2}y_{2}^{2} &=&0\text{.}  \notag
\end{eqnarray}
\end{proposition}

\begin{proof}
This follows directly from Proposition \ref{P1}(i) upon rewriting (\ref{2})
as a quadratic in $w$ and then letting $x=x_{j},y=y_{j},j=1,2$.
\end{proof}

\begin{lemma}
\label{L5}Suppose that $(w_{0},t_{0})$ is a solution of (\ref{3}).

(i) If $0\leq t_{0}\leq 1$, then $0\leq w_{0}\leq 1$ and (ii) If $0<t_{0}<1$
and $x_{1}(1-y_{2})-x_{2}(1-y_{1})\neq 0$, then $0<w_{0}<1$.
\end{lemma}

\begin{proof}
Solving each equation in (\ref{3}) separately for $w$ as a function of $t$
yields 
\begin{equation}
w=ty_{j}\dfrac{\allowbreak x_{j}+(1-2x_{j})t\pm 2\sqrt{t(1-t)}A_{j}}{%
\allowbreak q_{j}(t)},j=1,2\text{.}  \label{wt}
\end{equation}%
Suppose that $(w_{0},t_{0})$ is a solution of (\ref{3}). If $0\leq t_{0}\leq
1$, then we know that $w_{0}$ is real by looking at (\ref{wt}) with $t=t_{0}$%
. Now, one can also solve \ref{3}) for $t$ as a function of $w$ to obtain 
\begin{equation}
t=wx_{j}\dfrac{\allowbreak y_{j}+(1-2y_{j})w\pm 2\sqrt{w(1-w)}\sqrt{%
y_{j}(1-x_{j}-y_{j})}}{\allowbreak (y_{j}-w)^{2}+4x_{j}y_{j}w(1-w)},j=1,2%
\text{.}  \label{tw}
\end{equation}

Note that $\allowbreak $the quadratic in $w$,$\allowbreak \
(y_{j}-w)^{2}+4x_{j}y_{j}w(1-w)$, has negative discriminant and thus no real
roots. $\allowbreak $If $w_{0}\notin \lbrack 0,1]$, then (\ref{tw}) with $%
w=w_{0}$ shows that $t_{0}$ is \textbf{not} real, a contradiction. That
proves (i). Now suppose that $0<t_{0}<1$. If $w_{0}=0$, then by (\ref{tw}), $%
t_{0}=0$, so $0<w_{0}$. If $w_{0}=1$, then by (\ref{tw}), with $w=1$, $t_{0}=%
\tfrac{x_{j}}{1-y_{j}},j=1,2$. But then $\tfrac{x_{1}}{1-y_{1}}=\tfrac{x_{2}%
}{1-y_{2}}$, which implies that $x_{1}(1-y_{2})-x_{2}(1-y_{1})=0$. That
proves (ii).
\end{proof}

\begin{lemma}
\label{L6}Suppose that $x_{2}y_{1}-x_{1}y_{2}>0$. If $%
(1-x_{2})y_{1}-(1-x_{1})y_{2}>0$ or if $x_{2}(1-y_{1})-x_{1}(1-y_{2})<0$,
then $J>0$.
\end{lemma}

\begin{proof}
Suppose first that $(1-x_{2})y_{1}-(1-x_{1})y_{2}>0$. Then $%
x_{2}(1-x_{2}-y_{2})y_{1}^{2}-x_{1}(1-x_{1}-y_{1})y_{2}^{2}=$

$%
x_{2}(1-x_{2})y_{1}y_{1}-x_{2}y_{2}y_{1}^{2}-x_{1}(1-x_{1})y_{2}^{2}+x_{1}y_{1}y_{2}^{2}> 
$

$%
x_{2}(1-x_{1})y_{2}y_{1}-x_{2}y_{2}y_{1}^{2}-x_{1}(1-x_{1})y_{2}^{2}+x_{1}y_{1}y_{2}^{2}= 
$

$\allowbreak y_{2}\left( 1-x_{1}-y_{1}\right) \left(
x_{2}y_{1}-x_{1}y_{2}\right) >0$. Second, suppose that $%
x_{2}(1-y_{1})-x_{1}(1-y_{2})<0$; Since $1-y_{1}<\tfrac{x_{1}(1-y_{2})}{x_{2}%
}$, we have

$%
x_{2}(1-x_{2}-y_{2})y_{1}^{2}-x_{1}(1-x_{1}-y_{1})y_{2}^{2}>x_{2}(1-x_{2}-y_{2})y_{1}^{2}-x_{1}\left( 
\tfrac{x_{1}(1-y_{2})}{x_{2}}-x_{1}\right) y_{2}^{2}=\allowbreak \tfrac{%
\left( 1-x_{2}-y_{2}\right) \left( x_{2}y_{1}-x_{1}y_{2}\right)
(x_{2}y_{1}+x_{1}y_{2})}{x_{2}}\allowbreak >0$.
\end{proof}

\qquad We now prove the following lemma about connected sets defined by
inequalities involving continuous functions.

\begin{lemma}
\label{L7}Let $f(x_{1},...,x_{n-1})$ be a continuous function of $%
x_{1},...,x_{n-1}$ in $\bar{W}$, where $W$ is some connected domain in $%
R^{n-1}$. Then the sets

(i) $\left\{ (x_{1},...,x_{n-1},x_{n})\in W\times \Re
:x_{n}>f(x_{1},...,x_{n-1})\right\} $ and

(ii) $\left\{ (x_{1},...,x_{n-1},x_{n})\in W\times \Re
:x_{n}<f(x_{1},...,x_{n-1})\right\} $ are connected.
\end{lemma}

\begin{proof}
We prove (i), the proof of (ii) being very similar. Let

$H=\left\{ (x_{1},...,x_{n-1},x_{n})\in \bar{W}\times \Re :x_{n}\geq
f(x_{1},...,x_{n-1})\right\} $. Suppose that $P_{1}=(x_{1,1},...,x_{n,1})\in
H$ and $P_{2}=(x_{1,2},...,x_{n,2})\in H$. Let

$S=\left\{ (x_{1},...,x_{n-1},x_{n})\in \bar{W}\times \Re
:x_{n}=f(x_{1},...,x_{n-1})\right\} =$ graph of $f$. Let $L_{1}$ be the
vertical line in $R^{n}$ thru $P_{1}$, so that $L_{1}$ has parametric
equations $x_{1}=x_{1,1},...,x_{n-1}(t)=x_{n-1,1},x_{n}(t)=t,-\infty
<t<\infty $. Let $L_{2}$ be the vertical line thru $P_{2}\ $in $R^{n}$, so
that $L_{2}$ has parametric equations $%
x_{1}=x_{1,2},...,x_{n-1}(t)=x_{n-1,2},x_{n}(t)=t,-\infty <t<\infty $. Let $%
Q_{1}=(x_{1,1},...,x_{n-1,1},u_{n,1})\in S$ be the point of intersection of $%
L_{1}$ with $S$ and let $Q_{2}=(x_{1,2},...,x_{n-1,2},u_{n,2})\in S$ be the
point of intersection of $L_{2}$ with $S$. If $%
x_{n,1}>f(x_{1,1},...,x_{n-1,1})$, let $\Gamma _{1}$ be the path consisting
of the vertical line segment from $P_{1}$ to $S$, so that $\Gamma _{1}$ has
parametric equations $%
x_{1}=x_{1,1},...,x_{n-1}(t)=x_{n-1,1},x_{n}(t)=tu_{n,1}+(1-t)x_{n,1},0\leq
t\leq 1$. If $x_{n,1}=f(x_{1,1},...,x_{n-1,1})$, skip $\Gamma _{1}$. Let $%
\Gamma _{2}$ be any continuous path lying on $S$ from $Q_{1}$ to $Q_{2}$. If 
$x_{n,2}>f(x_{1,2},...,x_{n-1,2})$, let $\Gamma _{3}$ be the path consisting
of the vertical line segment from $Q_{2}$ to $P_{2}$, so that $\Gamma _{3}$
has parametric equations $%
x_{1}=x_{1,2},...,x_{n-1}(t)=x_{n-1,2},x_{n}(t)=(1-t)u_{n,2}+tx_{n,2},0\leq
t\leq 1$. If $x_{n,2}=f(x_{1,2},...,x_{n-1,2})$, skip $\Gamma _{3}$. Then $%
\Gamma _{1}\cup \Gamma _{2}\cup \Gamma _{3}$ is a continuous path from $%
P_{1} $ to $P_{2}$ which lies entirely in $H$, which implies that $H$ is
connected. Since a set is connected if and only if its closure is connected,
it follows immediately that $\left\{ (x_{1},...,x_{n-1},x_{n})\in W\times
\Re :x_{n}>f(x_{1},...,x_{n-1})\right\} $ is also connected.
\end{proof}

In many of the lemmas below and in the proof of Theorem \ref{T1}, one must
look at the case where $P_{1}=(x_{1},y_{1})$ and $P_{2}=(x_{2},y_{2})$
satisfy the equation $%
x_{2}(1-x_{2}-y_{2})y_{1}^{2}-x_{1}(1-x_{1}-y_{1})y_{2}^{2}=0$. Useful,
then, is the following:

\begin{eqnarray}
\text{Let }J &=&x_{2}(1-x_{2}-y_{2})y_{1}^{2}-x_{1}(1-x_{1}-y_{1})y_{2}^{2}%
\text{,}  \notag \\
A_{i} &=&\sqrt{x_{i}(1-x_{i}-y_{i})},i=1,2\text{.}  \label{28} \\
\text{Then }J &=&0\iff \dfrac{y_{1}}{y_{2}}=\dfrac{A_{1}}{A_{2}}\text{.} 
\notag
\end{eqnarray}

We now define the following very important value of $t$ if $%
2x_{2}y_{1}-2x_{1}y_{2}+y_{2}-y_{1}\neq 0$:

\begin{equation}
t_{0}=\dfrac{x_{2}y_{1}-x_{1}y_{2}}{2x_{2}y_{1}-2x_{1}y_{2}+y_{2}-y_{1}}%
\text{.}  \label{t0}
\end{equation}

\begin{lemma}
\label{L17}Let $t_{0}$ be given by (\ref{t0}) and suppose that $%
x_{2}y_{1}-x_{1}y_{2}>0$ and\ $(1-x_{1})y_{2}-(1-x_{2})y_{1}\neq 0$. If $%
(1-x_{1})y_{2}-(1-x_{2})y_{1}>0$ or if $J=\allowbreak 0$, then $%
2x_{2}y_{1}-2x_{1}y_{2}+y_{2}-y_{1}>0$ and $0<t_{0}<1$.
\end{lemma}

\begin{proof}
Suppose that $x_{2}y_{1}-x_{1}y_{2}>0$. If $(1-x_{1})y_{2}-(1-x_{2})y_{1}>0$%
, then $%
2x_{2}y_{1}-2x_{1}y_{2}+y_{2}-y_{1}=x_{2}y_{1}-x_{1}y_{2}+(1-x_{1})y_{2}-(1-x_{2})y_{1}>0 
$, and so $t_{0}>0$. Also, $1-t_{0}=\allowbreak \tfrac{%
(1-x_{1})y_{2}-(1-x_{2})y_{1}}{2x_{2}y_{1}-2x_{1}y_{2}+y_{2}-y_{1}}>0$. Now
suppose that $J=\allowbreak 0$. By Lemma \ref{L6}(i), $%
(1-x_{1})y_{2}-(1-x_{2})y_{1}>0$ and the conclusion follows by what we just
proved.
\end{proof}

For fixed $(x_{1},x_{2},y_{1},y_{2})$, we now define the following useful
polynomials in $t$:

\begin{eqnarray}
B(t) &=&y_{1}^{2}q_{2}(t)-y_{2}^{2}q_{1}(t)\text{ and}  \label{BC} \\
C(t) &=&y_{1}{\large (}x_{1}+(1-2x_{1})t{\large )}q_{2}(t)-y_{2}{\large (}%
x_{2}+(1-2x_{2})t{\large )}q_{1}(t)\text{.}  \notag
\end{eqnarray}

The proof of Theorem \ref{T1} would be somewhat simpler if one did not have
to allow for the possibility that $B(t)=C(t)=0$ can occur for the same value
of $0<t<1$. The purpose of the following lemmas is to give necessary and
sufficient conditions for $B(t)=C(t)=0$.

Using a computer algebra system, one has the following rational function\
identities, where $t_{0}$ is given by (\ref{t0}). 
\begin{gather}
\left( 2x_{2}y_{1}-2x_{1}y_{2}+y_{2}-y_{1}\right) ^{2}B(t_{0})=\allowbreak
\label{BCt0} \\
4\allowbreak J\left( x_{2}y_{1}-x_{1}y_{2}\right) {\large (}%
(1-x_{2})y_{1}-(1-x_{1})y_{2}{\large )}\text{, and}  \notag \\
\left( 2x_{2}y_{1}-2x_{1}y_{2}+y_{2}-y_{1}\right) ^{3}C(t_{0})=  \notag \\
-4J\left( x_{2}-x_{1}\right) \left( x_{2}y_{1}-x_{1}y_{2}\right) {\large (}%
(1-x_{2})y_{1}-(1-x_{1})y_{2}{\large )}\text{.}  \notag
\end{gather}

\begin{lemma}
\label{L13}Suppose that $x_{2}y_{1}-x_{1}y_{2}\neq 0\neq
(1-x_{2})y_{1}-(1-x_{1})y_{2}$ and let $t_{0}$ be given by (\ref{t0}).
Consider the system of equations in $t$,%
\begin{equation}
B(t)=0,C(t)=0\text{.}  \label{4}
\end{equation}%
If (\ref{4}) has a solution, then $2x_{2}y_{1}-2x_{1}y_{2}+y_{2}-y_{1}\neq 0$%
, that solution is unique and\textbf{\ }is given by $t=t_{0}$, and $%
J=\allowbreak 0$.
\end{lemma}

\begin{proof}
If $x_{2}+(1-2x_{2})t=0$, then $C(t)=0$ implies that $x_{1}+(1-2x_{1})t=0$
since $q_{2}$ has no real roots. The system of equations 
\begin{eqnarray*}
x_{1}+(1-2x_{1})t &=&0 \\
x_{2}+(1-2x_{2})t &=&0
\end{eqnarray*}%
yields $(2x_{2}-2x_{1})t=x_{2}-x_{1}$, so $t=\tfrac{1}{2}$ since $x_{1}\neq
x_{2}$. Substituting $t=\tfrac{1}{2}$ into either equation above gives $%
\tfrac{1}{2}=0$. Thus we may assume that $x_{2}+(1-2x_{2})t\neq 0$ if $t$ is
a solution of (\ref{4}). Similarly, we may assume that $x_{1}+(1-2x_{1})t%
\neq 0$ if $t$ is a solution of (\ref{4}). $B(t)=0$ implies that $\tfrac{%
y_{1}^{2}}{y_{2}^{2}}=\tfrac{q_{1}(t)}{q_{2}(t)}$ and $C(t)=0$ implies that $%
\tfrac{y_{1}{\large (}x_{1}+(1-2x_{1})t{\large )}}{y_{2}{\large (}%
x_{2}+(1-2x_{2})t{\large )}}=\tfrac{q_{1}(t)}{q_{2}(t)}$, and thus $\tfrac{%
y_{1}^{2}}{y_{2}^{2}}=$ $\tfrac{y_{1}{\large (}x_{1}+(1-2x_{1})t{\large )}}{%
y_{2}{\large (}x_{2}+(1-2x_{2})t{\large )}}$. That yields $\tfrac{y_{1}}{%
y_{2}}=\tfrac{x_{1}+(1-2x_{1})t}{x_{2}+(1-2x_{2})t}$, so $y_{1}{\large (}%
x_{2}+(1-2x_{2})t{\large )}-y_{2}{\large (}x_{1}+(1-2x_{1})t{\large )}=0$,
which implies that $\left( y_{1}-y_{2}+2x_{1}y_{2}-2x_{2}y_{1}\right)
t+x_{2}y_{1}-x_{1}y_{2}=0$; Thus $2x_{1}y_{2}-2x_{2}y_{1}+y_{1}-y_{2}\neq 0$%
, and solving for $t$ then yields $t=t_{0}$. (\ref{4}) then yields

$B(t_{0})=C(t_{0})=0$. Since $x_{2}y_{1}-x_{1}y_{2}\neq 0$ and $%
(1-x_{2})y_{1}-(1-x_{1})y_{2}\neq 0$ by assumption, by (\ref{BCt0}) we have $%
J=0$.
\end{proof}

\begin{lemma}
\label{L8}Suppose that $(x_{1},y_{1})\in G$ and for fixed $t,0<t<1$, let $%
g(w)=q_{1}(t)w^{2}+2ty_{1}{\large (}(2x_{1}-1)t-x_{1}{\large )}\allowbreak
w+y_{1}^{2}t^{2}$. Then $g$ has two distinct roots $w_{1},w_{2}\in (0,1]$.
\end{lemma}

\begin{proof}
Suppose that $0<t<1$; $g^{\prime }(w_{0})=0$, where $w_{0}=\tfrac{ty_{1}%
{\large (}x_{1}-(2x_{1}-1)t{\large )}}{q_{1}(t)}$ and $g(w_{0})=-\tfrac{%
4t^{3}x_{1}y_{1}^{2}\left( 1-x_{1}-y_{1}\right) (1-t)}{%
(t-x_{1})^{2}+4x_{1}y_{1}t(1-t)}<0$ since $(x_{1},y_{1})\in G$.

Case 1: $t\neq \tfrac{x_{1}}{1-y_{1}}$

Let $V=\left\{ (x_{1},y_{1})\in G:g\ \text{has two distinct roots\ in }%
(0,1)\right\} $; It follows easily that $V$ is an open subset of $G$ since $%
g(0)=y_{1}^{2}t^{2}>0$ and $g(1)={\large (}x_{1}-(1-y_{1})t{\large )}^{2}>0$%
. Since $g$ cannot have a double root and $g(0),g(1)>0$, $V$ is a closed
subset of $G$; Now let $x_{1}=0$ and $y_{1}=\tfrac{1}{2}$, which implies
that $g(w)=\allowbreak \tfrac{1}{4}t^{2}\left( 2w-1\right) ^{2}$ and thus $g$
has a double root at $w=\allowbreak \tfrac{1}{2}$; Then for $%
(x_{1},y_{1})\in G$ with $(x_{1},y_{1})$ close to $\left( 0,\tfrac{1}{2}%
\right) $, $g$ has two roots in $(0,1)$; They must be distinct as argued
above, which proves that $V$ is nonempty. Since $G$ is connected, $V=G$.

Case 2: $t=\tfrac{x_{1}}{1-y_{1}}$

Note that $\tfrac{x_{1}}{1-y_{1}}\in (0,1)$ since $(x_{1},y_{1})\in G$; $%
g(w)=\allowbreak x_{1}^{2}y_{1}\left( 1-w\right) \tfrac{\left(
4x_{1}+3y_{1}-4\right) w+y_{1}}{\left( 1-y_{1}\right) ^{2}}=0$ implies that $%
w=1$ or $w=\tfrac{y_{1}}{4-4x_{1}-3y_{1}}$; Also, $%
4-4x_{1}-3y_{1}>4-4x_{1}-4y_{1}=4(1-x_{1}-y_{1})>0$, which implies that $%
\tfrac{y_{1}}{4-4x_{1}-3y_{1}}>0$; Then $\tfrac{y_{1}}{4-4x_{1}-3y_{1}}%
<1\iff 4-4x_{1}-4y_{1}>0$ and so $\tfrac{y_{1}}{4-4x_{1}-3y_{1}}<1$.
\end{proof}

The next lemma shows that with certain assumptions on $P_{1}$ and $P_{2}$,
the two equations in (\ref{3}) are multiples of one another when $t=t_{0}$.

\begin{lemma}
\label{L11}Let $t_{0}$ be given by (\ref{t0}). Suppose that $%
x_{2}y_{1}-x_{1}y_{2}>0$ and\ $(1-x_{2})y_{1}-(1-x_{1})y_{2}\neq 0$. If $%
J=\allowbreak 0$, then

$q_{2}(t_{0})w^{2}+2y_{2}t_{0}{\large (}(2x_{2}-1)t_{0}-x_{2})\allowbreak 
{\large )}w+t_{0}^{2}y_{2}^{2}=\dfrac{y_{2}^{2}}{y_{1}^{2}}{\large (}%
q_{1}(t_{0})w^{2}+2y_{1}t_{0}{\large (}(2x_{1}-1)t_{0}-x_{1}{\large )}%
\allowbreak w+t_{0}^{2}y_{1}^{2}{\large )}$.
\end{lemma}

\begin{proof}
By Lemma \ref{L17}, $2x_{2}y_{1}-2x_{1}y_{2}+y_{2}-y_{1}\neq 0$. Now for any 
$y_{2},y_{1}^{2}t_{0}^{2}y_{2}^{2}=y_{2}^{2}t_{0}^{2}y_{1}^{2}$ is trivial
and it also follows easily that $y_{1}^{2}y_{2}t_{0}{\large (}%
(2x_{2}-1)t_{0}-x_{2})=y_{2}^{2}y_{1}t_{0}{\large (}(2x_{1}-1)t_{0}-x_{1}%
{\large )}$.

$J=\allowbreak 0$ implies that $\tfrac{y_{1}^{2}}{y_{2}^{2}}=\tfrac{%
x_{1}(1-x_{1}-y_{1})}{x_{2}(1-x_{2}-y_{2})}$, so $%
y_{1}^{2}q_{2}(t_{0})=y_{2}^{2}q_{1}(t_{0})\iff \tfrac{x_{1}(1-x_{1}-y_{1})}{%
x_{2}(1-x_{2}-y_{2})}=\tfrac{q_{1}(t_{0})}{q_{2}(t_{0})}\iff
x_{2}(1-x_{2}-y_{2})q_{1}(t_{0})-x_{1}(1-x_{1}-y_{1})q_{2}(t_{0})=0\iff
\allowbreak \tfrac{\left( x_{2}-x_{1}\right) ^{2}J}{\left(
2x_{2}y_{1}-2x_{1}y_{2}+y_{2}-y_{1}\right) ^{2}}=\allowbreak 0$.
\end{proof}

For fixed $(x_{1},x_{2},y_{1},y_{2})$, we now introduce the following
quadratic polynomials in $t$, which will be critical for the proof of
Theorem \ref{T1}. 
\begin{eqnarray}
R(t) &=&{\large (}4y_{1}y_{2}\left(
x_{1}y_{2}+x_{2}y_{1}+2x_{1}x_{2}-x_{2}-x_{1}\right)
-(y_{2}-y_{1})^{2}+8y_{1}y_{2}A_{1}A_{2}\allowbreak {\large )}t^{2}  \notag
\\
&&+2{\large (}\allowbreak x_{2}y_{1}^{2}+x_{1}y_{2}^{2}-y_{1}y_{2}\left(
2x_{1}y_{2}+2x_{2}y_{1}+4x_{1}x_{2}-x_{1}-x_{2}\right)
-4y_{1}y_{2}A_{1}A_{2}\allowbreak {\large )}t  \notag \\
&&-\left( x_{2}y_{1}-x_{1}y_{2}\right) ^{2},j=1,2\text{, }  \label{R}
\end{eqnarray}

and 
\begin{eqnarray}
S(t) &=&{\large (}4y_{1}y_{2}\left(
x_{1}y_{2}+x_{2}y_{1}+2x_{1}x_{2}-x_{2}-x_{1}\right)
-(y_{2}-y_{1})^{2}-8y_{1}y_{2}A_{1}A_{2}\allowbreak {\large )}t^{2}  \notag
\\
&&+2{\large (}\allowbreak x_{2}y_{1}^{2}+x_{1}y_{2}^{2}-y_{1}y_{2}\left(
2x_{1}y_{2}+2x_{2}y_{1}+4x_{1}x_{2}-x_{1}-x_{2}\right)
+4y_{1}y_{2}A_{1}A_{2}\allowbreak {\large )}t  \notag \\
&&-\left( x_{2}y_{1}-x_{1}y_{2}\right) ^{2},j=1,2\text{.}  \label{S}
\end{eqnarray}

Since $R(t)-S(t)=-16A_{1}A_{2}y_{1}y_{2}t(1-t)$, $R$ and $S$ cannot take the
same value if $t\in (0,1)$. In particular, 
\begin{equation}
R(t_{1})=0=S(t_{2}),0<t_{1},t_{2}<1\Rightarrow t_{1}\neq t_{2}\text{.}
\label{ReqS}
\end{equation}

Note that 
\begin{eqnarray}
R(0) &=&S(0)=-\left( x_{2}y_{1}-x_{1}y_{2}\right) ^{2}  \label{endptsRS} \\
R(1) &=&S(1)=\allowbreak -{\large (}(1-x_{2})y_{1}-(1-x_{1})y_{2}{\large )}%
^{2}\text{.}  \notag
\end{eqnarray}

\begin{lemma}
\label{L16}: If $R(0)=0$, then $R^{\prime }(0)\neq 0$ and if $R(1)=0$, then $%
R^{\prime }(1)\neq 0$.
\end{lemma}

\begin{proof}
If $R(0)=0$, then $y_{2}=\tfrac{x_{2}y_{1}}{x_{1}}$, which implies that $%
x_{1}R^{\prime }(0)=\allowbreak
-4x_{2}y_{1}^{2}(-x_{1}-x_{2}+2x_{2}y_{1}+2x_{1}x_{2}+2\sqrt{1-x_{1}-y_{1}}%
\sqrt{-x_{2}(-x_{1}+x_{1}x_{2}+x_{2}y_{1})})$, so $R^{\prime }(0)=0$ implies
that

$%
(2x_{2}y_{1}+2x_{1}x_{2}-x_{1}-x_{2})^{2}-4(1-x_{1}-y_{1})(-x_{2}(-x_{1}+x_{1}x_{2}+x_{2}y_{1}))=\allowbreak \left( x_{2}-x_{1}\right) ^{2}=0 
$, which contradicts (\ref{1}). If $R(1)=0$, then $y_{2}=\tfrac{%
(1-x_{2})y_{1}}{1-x_{1}}$, which implies that

$\left( x_{1}-1\right) ^{2}R^{\prime }(1)=\allowbreak 4y_{1}^{2}\left(
x_{2}-1\right) \times $

$\left( \allowbreak \left( 1-x_{1}-y_{1}\right) \left( \allowbreak
x_{1}+x_{2}-2x_{1}x_{2}\right) +2(x_{1}-1)A_{1}\sqrt{\left(
1-x_{1}-y_{1}\right) x_{2}\tfrac{x_{2}-1}{x_{1}-1}}\right) $, so $R^{\prime
}(1)=0$ implies that

${\large (}\left( 1-x_{1}-y_{1}\right) \left( \allowbreak
x_{1}+x_{2}-2x_{1}x_{2}\right) {\large )}^{2}$

$-4\left( x_{1}-1\right) ^{2}x_{1}\left( 1-x_{1}-y_{1}\right) \left(
1-x_{1}-y_{1}\right) x_{2}\tfrac{x_{2}-1}{x_{1}-1}=\allowbreak \left(
x_{1}+y_{1}-1\right) ^{2}\left( x_{2}-x_{1}\right) ^{2}=0$, which again
contradicts (\ref{1}).
\end{proof}

\begin{lemma}
\label{L18}Let $D_{R}$ and $D_{S}$ denote the discriminants of $R$ and of $S$%
, respectively. Then $D_{R}\geq 0$ and $D_{S}>0$, and thus $R$ and $S$ each
have real zeros.
\end{lemma}

\begin{proof}
Using (\ref{R}) we have 
\begin{gather*}
\tfrac{1}{4}D_{R}={\large (}\allowbreak \allowbreak
x_{2}y_{1}^{2}+x_{1}y_{2}^{2}-y_{1}y_{2}\left(
2x_{1}y_{2}+2x_{2}y_{1}+4x_{1}x_{2}-x_{1}-x_{2}\right)
-4y_{1}y_{2}A_{1}A_{2}^{2}{\large )}^{2} \\
+{\large (}4y_{1}y_{2}\left(
x_{1}y_{2}+x_{2}y_{1}+2x_{1}x_{2}-x_{2}-x_{1}\right)
-(y_{2}-y_{1})^{2}+8y_{1}y_{2}A_{1}A_{2}\allowbreak {\large )}\left(
x_{1}y_{2}-x_{2}y_{1}\right) ^{2}\text{.}
\end{gather*}

After some simplification we have $\tfrac{1}{16y_{1}y_{2}}%
D_{R}=U+A_{1}A_{2}V $, where 
\begin{gather*}
U=4x_{1}x_{2}y_{1}y_{2}(1-x_{1}-y_{1})(1-x_{2}-y_{2}) \\
-\left( x_{1}y_{2}+x_{2}y_{1}+2x_{1}x_{2}-x_{2}-x_{1}\right) {\large (}%
x_{1}y_{2}^{2}(1-x_{1}-y_{1})+x_{2}y_{1}^{2}(1-x_{2}-y_{2}){\large )} \\
V=\allowbreak 2\left( x_{2}y_{1}+x_{1}y_{2}\right) \left(
x_{2}y_{1}+x_{1}y_{2}+2y_{1}y_{2}-y_{1}-y_{2}\right) \text{.}
\end{gather*}

Using (\ref{S}), one can also show that $\tfrac{1}{16y_{1}y_{2}}%
D_{S}=U-A_{1}A_{2}V$. Since $x_{1}y_{2}+x_{2}y_{1}+2x_{1}x_{2}-x_{2}-x_{1}<0$
and $x_{1}y_{2}+x_{2}y_{1}+2y_{1}y_{2}-y_{1}-y_{2}<0$ by Lemma \ref{L12}, $%
U>0$ and $V<0$. Hence $U-A_{1}A_{2}V>0$ and so $D_{S}>0$. Some more
simplification yields $U+A_{1}A_{2}V=\tfrac{(U+A_{1}A_{2}V)(U-A_{1}A_{2}V)}{%
U-A_{1}A_{2}V}=\tfrac{U^{2}-x_{1}(1-x_{1}-y_{1})x_{2}(1-x_{2}-y_{2})V^{2}}{%
U-A_{1}A_{2}V}=\tfrac{{\large (}x_{2}(1-y_{1})-x_{1}(1-y_{2}){\large )}%
^{2}J^{2}}{U-A_{1}A_{2}V}$, which implies that

\begin{equation}
\dfrac{1}{16y_{1}y_{2}}D_{R}=\dfrac{{\large (}x_{2}(1-y_{1})-x_{1}(1-y_{2})%
{\large )}^{2}J^{2}}{U-A_{1}A_{2}V}\text{,}  \label{DR}
\end{equation}%
and thus $D_{R}\geq 0$.
\end{proof}

\begin{lemma}
\label{L9}(i) $R$ and $S$ are each concave down quadratic polynomials.

(ii) $R^{\prime }(0)>0,R^{\prime }(1)<0,S^{\prime }(0)>0$, and $S^{\prime
}(1)<0$.
\end{lemma}

\begin{proof}
Proving (ii) directly seems a bit complicated, so instead we use a
continuity/connected sets approach. First we show that the leading
coefficients of $R$ and $S$, $a_{R}=4y_{1}y_{2}\left(
x_{1}y_{2}+x_{2}y_{1}+2x_{1}x_{2}-x_{2}-x_{1}\right)
-(y_{2}-y_{1})^{2}+8y_{1}y_{2}A_{1}A_{2}$ and $a_{S}=4y_{1}y_{2}\left(
x_{1}y_{2}+x_{2}y_{1}+2x_{1}x_{2}-x_{2}-x_{1}\right)
-(y_{2}-y_{1})^{2}-8y_{1}y_{2}A_{1}A_{2}$, are negative. Since $%
x_{1}y_{2}+x_{2}y_{1}+2x_{1}x_{2}-x_{2}-x_{1}<0$ by Lemma \ref{L12}, $%
a_{S}<0 $. Now $a_{R}=\tfrac{p(x_{2})}{a_{S}}$, where 
\begin{gather*}
p(x_{2})=16y_{1}^{2}y_{2}^{2}\left( 1-y_{1}\right) ^{2}x_{2}^{2} \\
-8y_{1}y_{2}\times {\large (}%
4y_{1}^{2}x_{1}y_{2}^{2}+y_{1}y_{2}^{2}-y_{2}^{2}-4y_{1}x_{1}y_{2}^{2}+2x_{1}y_{2}^{2}+2y_{1}y_{2}
\\
-2y_{1}^{2}y_{2}-4y_{1}^{2}y_{2}x_{1}+y_{1}^{3}-y_{1}^{2}+2x_{1}y_{1}^{2}%
{\large )}x_{2} \\
+\left(
-4y_{1}x_{1}y_{2}^{2}+4y_{1}y_{2}x_{1}+y_{2}^{2}-2y_{1}y_{2}+y_{1}^{2}%
\right) ^{2}\text{.}
\end{gather*}

$\allowbreak $

$\allowbreak $ The discriminant of $p_{2}$, after factoring, equals $%
\allowbreak 256x_{1}y_{1}^{2}y_{2}^{2}\left( 2y_{1}y_{2}-y_{2}-y_{1}\right)
^{2}\left( y_{2}-y_{1}\right) ^{2}\allowbreak \left( x_{1}+y_{1}-1\right) <0$
by Lemma \ref{L12}. Since $p$ has no real roots and the leading coefficient
of $p$ is positive, $p(x_{2})>0$ for all $x_{2}\in \Re $, which implies that 
$a_{R}<0$. That proves (i). To prove (ii), let $H$ be the set of points in $%
\Re ^{4}$ satisying (\ref{1})--that is, $H=\left\{
(x_{1},x_{2},y_{1},y_{2})\in \Re ^{4}:0<x_{1}\neq x_{2},y_{1}\neq
y_{2}<1,x_{1}+y_{1}<1,x_{2}+y_{2}<1\right\} $. Note that $H$ is not
connected, so we consider the following connected subsets of $H$: $%
W_{1}=\left\{ (x_{1},x_{2},y_{1},y_{2})\in H:x_{1}<x_{2},y_{1}<y_{2}\right\}
,$

$W_{2}=\left\{ (x_{1},x_{2},y_{1},y_{2})\in
H:x_{2}<x_{1},y_{1}<y_{2}\right\} ,$

$W_{3}=\left\{ (x_{1},x_{2},y_{1},y_{2})\in
H:x_{2}<x_{1},y_{2}<y_{1}\right\} $, and

$W_{4}=\left\{ (x_{1},x_{2},y_{1},y_{2})\in
H:x_{1}<x_{2},y_{2}<y_{1}\right\} $. Let

$W_{R,j}=\left\{ (x_{1},x_{2},y_{1},y_{2})\in W_{j}\text{: }R^{\prime
}(0)>0\right\} ,j=1,2,3,4$. It is immediate that $W_{R,j}$ is an open subset
of $W_{j}$. To prove that $W_{R,j}$ is a closed subset of $W_{j}$, for fixed 
$j$, let $\left\{ (x_{1k},x_{2k},y_{1k},y_{2k})\right\} $ be a sequence of
points in $W_{j}$ converging to $(x_{1},x_{2},y_{1},y_{2})$, and let $%
\left\{ R_{k}\right\} $ be the corresponding sequence of quadratics given by
(\ref{R}) with $(x_{1k},x_{2k},y_{1k},y_{2k})$ substituted for $%
(x_{1},x_{2},y_{1},y_{2})$. Let $R$ be given by (\ref{R}), so that $\left\{
R_{k}\right\} $ converges to $R$. Since $R_{k}^{\prime }(0)>0,R^{\prime
}(0)\geq 0$. If $R^{\prime }(0)=0$, then $R(0)\neq 0$ by Lemma \ref{L16},
and thus $R(0)<0$ by (\ref{endptsRS}). But then $R$ has a local maximum at $%
0 $ and is negative at $0$, which would imply that $R$ has no real roots,
contradicting Lemma \ref{L18}. Hence $R^{\prime }(0)>0$, which proves that $%
W_{R,j}$ is a closed subset of $W_{j}$. To prove that $W_{j}$ is nonempty,
choose a point in each of the $W_{j}$'s, $j=1,2,3,4$; $\left( \tfrac{1}{8},%
\tfrac{1}{4},\tfrac{1}{6},\tfrac{1}{2}\right) $ implies that $R^{\prime
}(0)=\allowbreak \tfrac{1}{12}-\tfrac{\sqrt{51}}{144}>0$; $\left( \tfrac{1}{4%
},\tfrac{1}{6},\tfrac{1}{6},\tfrac{1}{3}\right) $implies that $R^{\prime
}(0)=\allowbreak \tfrac{11}{162}-\tfrac{\sqrt{7}}{54}>0$; $\left( \tfrac{1}{4%
},\tfrac{1}{6},\tfrac{1}{2},\tfrac{1}{4}\right) $implies that $R^{\prime
}(0)=\tfrac{1}{8}-\tfrac{\sqrt{2}}{12}>0$; $\left( \tfrac{1}{3},\tfrac{1}{2},%
\tfrac{1}{2},\tfrac{1}{4}\right) $implies that $R^{\prime }(0)=\allowbreak 
\tfrac{1}{12}>0$. Since $W_{j}$ is connected and $W_{R,j}$ is an open,
closed, and nonempty subset of $W_{j}$, $W_{R,j}=W_{j}$. That proves that $%
R^{\prime }(0)>0$ for any $(x_{1},x_{2},y_{1},y_{2})\in H$. Similarly one
can show that $R^{\prime }(1)<0$ for any $(x_{1},x_{2},y_{1},y_{2})\in H$.
Finally, $S^{\prime }(t)=R^{\prime }(t)-16A_{1}A_{2}y_{1}y_{2}\allowbreak
\left( 2t-1\right) $, which implies that $S^{\prime }(0)=R^{\prime
}(0)+16A_{1}A_{2}y_{1}y_{2}>0$ and $S^{\prime }(1)=R^{\prime
}(1)-16A_{1}A_{2}y_{1}y_{2}<0$.
\end{proof}

\begin{proposition}
\label{P4}(i) $R$ has two roots in $[0,1]$ counting multiplicities and $S$
has two distinct roots in $[0,1]$.

(ii) Suppose that $x_{2}(1-y_{1})-x_{1}(1-y_{2})\neq 0\neq $ $J$. Then $R$
has two distinct roots in $[0,1]$.

(iii) If $J=\allowbreak 0$, then $R(t_{0})=R^{\prime }(t_{0})=0$, where $%
t_{0}$ is given by (\ref{t0}).
\end{proposition}

\begin{proof}
(i) follows immediately from Lemma \ref{L9}, the fact that $R$ and $S$ each
have real zeros(Lemma \ref{L18}), and the fact that $R$ and $S$ are
non--positive at the endpoints of $[0,1]$ (\ref{endptsRS}). Now suppose that 
$x_{2}(1-y_{1})-x_{1}(1-y_{2})\neq 0\neq $ $J$. By (\ref{DR}) $D_{R}>0$ and
hence the roots of $R$ are distinct. (ii) then follows from (i). To prove
(iii), suppose that $J=0$. By (\ref{28}), $A_{2}=\left( \tfrac{y_{2}}{y_{1}}%
\right) A_{1}$ implies that $%
A_{1}A_{2}y_{1}y_{2}=A_{1}^{2}y_{2}^{2}=x_{1}(1-x_{1}-y_{1})y_{2}^{2}$. A
little simplification shows that $R(t_{0})=4\left(
x_{2}y_{1}-x_{1}y_{2}\right) {\large (}(1-x_{2})y_{1}-(1-x_{1})y_{2}{\large )%
}\allowbreak D(x_{1},x_{2},y_{1},y_{2})$, where

$D(x_{1},x_{2},y_{1},y_{2})=2A_{1}A_{2}y_{1}y_{2}-\allowbreak
x_{2}y_{1}^{2}(1-x_{2}-y_{2})-x_{1}y_{2}^{2}(1-x_{1}-y_{1})=2x_{1}(1-x_{1}-y_{1})y_{2}^{2}-\allowbreak x_{2}y_{1}^{2}(1-x_{2}-y_{2})-x_{1}y_{2}^{2}(1-x_{1}-y_{1})=x_{1}(1-x_{1}-y_{1})y_{2}^{2}-\allowbreak x_{2}y_{1}^{2}(1-x_{2}-y_{2})=-J=0 
$. Since $R^{\prime }(t_{0})=\allowbreak 4\left( y_{2}-y_{1}\right) \tfrac{%
D(x_{1},x_{2},y_{1},y_{2})}{2x_{1}y_{2}-2x_{2}y_{1}+y_{1}-y_{2}},R^{\prime
}(t_{0})=0$ as well.
\end{proof}

Using a computer algebra system, one has the following identity involving
the polynomials $B,C$, and $G=RS$:

\begin{eqnarray}
&&q_{j}(t)B^{2}(t)+4y_{j}{\large (}(2x_{j}-1)t-x_{j}{\large )}%
B(t)C(t)+4y_{j}^{2}C^{2}(t)  \label{BCG} \\
&=&{\large (}\left( 1-4x_{j}y_{j}\right) t^{2}-2x_{j}\left( 1-2y_{j}\right)
t+x_{j}^{2}{\large )}G(t),j=1,2\text{.}  \notag
\end{eqnarray}

\begin{proposition}
\label{P5}Let $G(t)=R(t)S(t)$, where $R$ and $S$ are given by (\ref{R}) and (%
\ref{S}), and let $t_{0}$ be given by (\ref{t0}). Suppose that $%
x_{2}y_{1}-x_{1}y_{2}\neq 0\neq (1-x_{2})y_{1}-(1-x_{1})y_{2}$. .
\end{proposition}

(i) If $G(t_{1})=0$, and $J\neq \allowbreak 0$ or $t_{1}\neq t_{0}$, then $%
C(t_{1})\neq 0$ and $(w_{1},t_{1})$ is a solution of (\ref{3}), where $w_{1}=%
\dfrac{t_{1}}{2}\dfrac{B(t_{1})}{C(t_{1})}$.

(ii) If $(w_{1},t_{1})$ is a solution of (\ref{3}) and $J\neq \allowbreak 0$
or $t_{1}\neq t_{0}$, then $C(t_{1})\neq 0,w_{1}=\dfrac{t_{1}}{2}\dfrac{%
B(t_{1})}{C(t_{1})}$, and $G(t_{1})=0$.

\begin{proof}
For any $t_{1}$ with $C(t_{1})\neq 0$, letting $t=t_{1},w=\tfrac{t_{1}}{2}%
\tfrac{B(t_{1})}{C(t_{1})}$ in the left hand side of each equation in (\ref%
{3}) yields

$\dfrac{t_{1}^{2}}{4}q_{j}(t_{1})\left( \tfrac{B(t_{1})}{C(t_{1})}\right)
^{2}+t_{1}^{2}y_{j}{\large (}(2x_{j}-1)t_{1}-x_{j}{\large )}\allowbreak
\left( \tfrac{B(t_{1})}{C(t_{1})}\right) +t_{1}^{2}y_{j}^{2}=\tfrac{t_{1}^{2}%
}{4{\large (}C(t_{1}){\large )}^{2}}{\large (}q_{j}(t_{1})B^{2}(t_{1})+4y_{j}%
{\large (}(2x_{j}-1)t_{1}-x_{j}{\large )}%
B(t_{1})C(t_{1})+4y_{j}^{2}C^{2}(t_{1}){\large )},j=1,2$. Using (\ref{BCG}),
we then have 
\begin{gather}
q_{j}(t_{1})w_{1}^{2}+2ty_{j}{\large (}(2x_{j}-1)t_{1}-x_{j}{\large )}%
\allowbreak w_{1}+t_{1}^{2}y_{j}^{2}=  \notag \\
\dfrac{t_{1}^{2}}{4{\large (}C(t_{1}){\large )}^{2}}\allowbreak {\large (}%
\left( 1-4x_{j}y_{j}\right) t_{1}^{2}-2x_{j}\left( 1-2y_{j}\right)
t_{1}+x_{j}^{2}{\large )}G(t_{1}),  \label{60} \\
j=1,2\text{\ if }w_{1}=\dfrac{t_{1}}{2}\dfrac{B(t_{1})}{C(t_{1})}%
,C(t_{1})\neq 0\text{.}  \notag
\end{gather}

Now suppose that $G(t_{1})=0$, and $J\neq \allowbreak 0$ or $t_{1}\neq t_{0}$%
. By (\ref{BCG}), $q_{j}(t_{1})B^{2}(t_{1})+4y_{j}{\large (}%
(2x_{j}-1)t_{1}-x_{j}{\large )}B(t_{1})C(t_{1})+4y_{j}^{2}C^{2}(t_{1})=%
{\large (}\left( 1-4x_{j}y_{j}\right) t_{1}^{2}-2x_{j}\left( 1-2y_{j}\right)
t_{1}+x_{j}^{2}{\large )}G(t_{1})=0,j=1,2$. If $C(t_{1})=0$, then $%
q_{j}(t_{1})B^{2}(t_{1})=0$, so $B(t_{1})=0$, which contradicts Lemma \ref%
{L13} since one cannot have $B(t_{1})=C(t_{1})=0$ if $J\neq \allowbreak 0$
or if $t_{1}\neq t_{0}$, and so $C(t_{1})\neq 0$. Letting $t=t_{1},w=\tfrac{%
t_{1}}{2}\tfrac{B(t_{1})}{C(t_{1})}$ in the left hand side of each equation
in (\ref{3}) then yields $0$ by (\ref{60}). That proves (i). Now suppose
that $(w_{1},t_{1})$ is a solution of (\ref{3}) and $J\neq \allowbreak 0$ or 
$t_{1}\neq t_{0}$, which implies that $q_{j}(t_{1})w_{1}^{2}+2ty_{j}{\large (%
}(2x_{j}-1)t_{1}-x_{j}{\large )}\allowbreak w_{1}+t_{1}^{2}y_{j}^{2}=0,j=1,2$%
. Eliminate $w^{2}$ in (\ref{3}) by multiplying the 1st equation in (\ref{3}%
) thru by $\allowbreak q_{2}(t)$ and by multiplying the 2nd equation in (\ref%
{3}) thru by $\allowbreak q_{1}(t)$ to obtain 
\begin{eqnarray}
q_{1}(t)q_{2}(t)w^{2}+2ty_{1}{\large (}(2x_{1}-1)t-x_{1}{\large )}%
q_{2}(t)w+t^{2}y_{1}^{2}q_{2}(t) &=&0  \label{12} \\
q_{1}(t)q_{2}(t)w^{2}+2ty_{2}{\large (}(2x_{2}-1)t-x_{2}{\large )}%
\allowbreak q_{1}(t)w+t^{2}y_{2}^{2}q_{1}(t) &=&0\text{.}  \notag
\end{eqnarray}

Let $t=t_{1}$ and $w=w_{1}$ in (\ref{12}) and then subtract and simplify to
obtain $-2C(t_{1})w_{1}+t_{1}B(t_{1})=0$. If $C(t_{1})=0$ then $B(t_{1})=0$
as well, which again contradicts Lemma \ref{L13}. Hence $C(t_{1})\neq 0$,
which implies that $w_{1}=\tfrac{t_{1}}{2}\tfrac{B(t_{1})}{C(t_{1})}$. By (%
\ref{60}), ${\large (}\left( 1-4x_{j}y_{j}\right) t_{1}^{2}-2x_{j}\left(
1-2y_{j}\right) t_{1}+x_{j}^{2}{\large )}G(t_{1})=0,j=1,2$. The quadratic $%
\left( 1-4x_{j}y_{j}\right) t^{2}-2x_{j}\left( 1-2y_{j}\right) t+x_{j}^{2}$
has no real roots since its discriminant is $4x_{j}^{2}\left(
1-2y_{j}\right) ^{2}-4\left( 1-4x_{j}y_{j}\right) x_{1}^{2}=\allowbreak
16x_{j}^{2}y_{j}\left( x_{j}+y_{j}-1\right) <0$, and hence $G(t_{1})=0$.
That proves (ii). \textbf{\ }
\end{proof}

\subsection{Proof of Theorem \protect\ref{T1}}

Throughout this section we assume that $T$ is the unit triangle --the
triangle with vertices $(0,0),(1,0)$, and $(0,1)$, and that that (\ref{1})
holds throughout.

\subsubsection{Proof of (i)}

\begin{proof}
Suppose that $x_{2}y_{1}-x_{1}y_{2}\neq 0,(1-x_{2})y_{1}-(1-x_{1})y_{2}\neq
0 $, and $x_{2}(1-y_{1})-x_{1}(1-y_{2})\neq 0$. Now we must consider two
cases.

\textbf{Case 1: }$J\neq \allowbreak 0$

Then $R$ and $S$ each have two distinct roots in $(0,1)$ by Proposition \ref%
{P4}(i) and (ii) and (\ref{endptsRS}). That gives a total of four distinct
roots, $0<t_{1},t_{2},t_{3},t_{4}<1$ for $G(t)=R(t)S(t)$ by (\ref{ReqS}).

By Proposition \ref{P5}(i), $C(t_{i})\neq 0,i=1,...,4$, and if $w_{i}=\tfrac{%
t_{i}}{2}\tfrac{B(t_{i})}{C(t_{i})},i=1,...,4$, then $(w_{i},t_{i})$ is a
solution of (\ref{3}) and thus we have four distinct solutions of (\ref{3}).
Also, by Proposition \ref{P5}(ii), there are no other solutions of (\ref{3})
since $G$ has precisely the four roots $t_{1},t_{2},t_{3},t_{4}$. \ By Lemma %
\ref{L5}, $0<w_{i}<1,i=1,...,4$. By Proposition \ref{P2}, there are
precisely four distinct ellipses inscribed in $T$ which pass through $P_{1}$
and $P_{2}$.

\textbf{Case 2: }$J=0$

By reordering the points $P_{1}$ and $P_{2}$, if necessary, we may assume
that 
\begin{equation}
x_{2}y_{1}-x_{1}y_{2}>0\text{.}  \label{5}
\end{equation}

Again $S$ has two distinct roots, $t_{1}\neq t_{2}$, in $(0,1)$ by
Proposition \ref{P4}(i) and (\ref{endptsRS}). By Lemma \ref{L17}, $%
2x_{2}y_{1}-2x_{1}y_{2}+y_{2}-y_{1}\neq 0$ and $0<t_{0}<1$, where $t_{0}$ is
given by (\ref{t0}). By Proposition \ref{P4}(iii), $R(t_{0})=R^{\prime
}(t_{0})=0$, and by (\ref{ReqS}), $t_{1}\neq t_{0}\neq t_{2}$. As for case 1
above, by Proposition \ref{P5}(i), letting $w_{i}=\tfrac{t_{i}}{2}\tfrac{%
B(t_{i})}{C(t_{i})},i=1,2$, yields two distinct solutions of (\ref{3}). Now
let $g(w)=q_{1}(t_{0})w^{2}+2t_{0}y_{1}{\large (}(2x_{1}-1)t_{0}-x_{1}%
{\large )}\allowbreak w+t_{0}^{2}y_{1}^{2}$. Then $g$ has two distinct roots 
$w_{3},w_{4}\in (0,1]$ by Lemma \ref{L8}, which implies that $%
q_{1}(t_{0})w_{i}^{2}+2t_{0}y_{1}{\large (}(2x_{1}-1)t_{0}-x_{1}{\large )}%
\allowbreak w_{i}+t_{0}^{2}y_{1}^{2}=0,i=3,4$. By Lemma \ref{L11}, $%
q_{2}(t_{0})w_{i}^{2}+2t_{0}y_{2}{\large (}(2x_{2}-1)t_{0}-x_{2})\allowbreak 
{\large )}w_{i}+t_{0}^{2}y_{2}^{2}=0$ and hence $(w_{3},t_{0})$ and $%
(w_{4},t_{0})$ are two distinct solutions of (\ref{3}). By Lemma \ref{L5}, $%
w_{3},w_{4}\in (0,1)$and by Proposition \ref{P5}(ii), there are no other
solutions of (\ref{3}) since $G$ has precisely the two roots $t_{1}\neq
t_{2} $ and a double root at $t_{0}$. By Proposition \ref{P2}, there are
precisely four distinct ellipses inscribed in $T$ which pass through $P_{1}$
and $P_{2} $.
\end{proof}

$\allowbreak $\textbf{Examples: }(1) $x_{1}=\tfrac{1}{4}$, $y_{1}=\tfrac{1}{8%
}$, $x_{2}=\tfrac{1}{2}$, $y_{2}=\tfrac{1}{6}$; Since $x_{2}y_{1}-x_{1}y_{2}%
\neq 0,(1-x_{2})y_{1}-(1-x_{1})y_{2}\neq 0$, and $%
x_{2}(1-y_{1})-x_{1}(1-y_{2})\neq 0$, there are four distinct ellipses
inscribed in $T$ which pass through $P_{1}=\left( \tfrac{1}{4},\tfrac{1}{8}%
\right) $ and $P_{2}=\left( \tfrac{1}{2},\tfrac{1}{6}\right) $; $%
J=\allowbreak -\dfrac{1}{576}\neq 0$; $R(t)=\allowbreak \allowbreak \left( 
\tfrac{1}{288}\sqrt{60}-\tfrac{5}{144}\right) t^{2}+\left( \tfrac{1}{32}-%
\tfrac{1}{288}\sqrt{60}\right) \allowbreak t-\tfrac{1}{2304}$, which has
roots $-\tfrac{1}{80}\sqrt{60}+\tfrac{3}{8}\pm \left( \tfrac{1}{20}\sqrt{10}-%
\tfrac{1}{8}\sqrt{6}\right) \approx \allowbreak 0.13$ and $0.43$

and $S(t)=\left( -\tfrac{1}{288}\sqrt{60}-\tfrac{5}{144}\right) t^{2}+\left( 
\tfrac{1}{32}+\tfrac{1}{288}\sqrt{60}\right) \allowbreak t-\tfrac{1}{2304}$,
which has roots $\tfrac{1}{80}\sqrt{60}+\tfrac{3}{8}\pm \left( \tfrac{1}{20}%
\sqrt{10}+\tfrac{1}{8}\sqrt{6}\right) \approx 0.008$ and $\allowbreak 0.94$

The four distinct solutions of (\ref{3}) are $\approx (0.43,\allowbreak
0.74),(0.13,0.03),(0.008,0.003)$, and $(\allowbreak 0.94\allowbreak ,0.22)$.

$\allowbreak $(2) $x_{1}=\tfrac{1}{8}$, $x_{2}=\tfrac{1}{4}$, $y_{1}=-\tfrac{%
1}{4}+\tfrac{1}{\sqrt{2}}$, $y_{2}=\tfrac{1}{2}$; Since $%
x_{2}y_{1}-x_{1}y_{2}\neq 0,(1-x_{2})y_{1}-(1-x_{1})y_{2}\neq 0$, and $%
x_{2}(1-y_{1})-x_{1}(1-y_{2})\neq 0$, there are four distinct ellipses
inscribed in $T$ which pass through $P_{1}=\left( \tfrac{1}{8},-\tfrac{1}{4}+%
\tfrac{1}{\sqrt{2}}\right) $ and $P_{2}=\left( \tfrac{1}{4},\tfrac{1}{2}%
\right) $; $J=\allowbreak 0$ and $R(t)=\tfrac{1}{128}\left( -3+2\sqrt{2}%
\right) \left( 4t-\sqrt{2}\right) ^{2}$, which has a double root at $%
t_{0}=\allowbreak \tfrac{\sqrt{2}}{4}\approx \allowbreak 0.35$; $%
S(t)=\allowbreak \left( \tfrac{1}{\sqrt{2}}-\tfrac{15}{16}\right) t^{2}+%
\tfrac{1}{16}\left( 5-\sqrt{2}\right) t+\dfrac{1}{64}(2\sqrt{2}-3)$, which
has roots $\tfrac{59}{194}+\tfrac{25}{194}\sqrt{2}\pm \tfrac{1}{194}\sqrt{%
3470+3532\sqrt{2}}\approx 0.01$ and $0.96$. From Lemma \ref{L8}, $%
g(w)=\allowbreak -\tfrac{1}{64\left( 2-\sqrt{2}\right) ^{2}}{\large (}%
\allowbreak \left( 120\sqrt{2}-172\right) w^{2}+\left( 26\sqrt{2}-34\right)
w+\allowbreak 30\sqrt{2}-43{\large )}\allowbreak $, which has roots $\tfrac{1%
}{4}+\tfrac{1}{4}\sqrt{2}\pm \tfrac{1}{4}\sqrt{2\sqrt{2}-1}\approx 0.27$ and 
$0.94$.

The four distinct solutions of (\ref{3}) are $\approx
(0.35,0.27),(0.35,0.94),(0.01,0.03)$, and $(0.96,\allowbreak 0.58)$.

\subsubsection{Proof of (ii)}

\begin{proof}
As in the proof of Theorem \ref{T1}(i) above, by reordering the points $%
P_{1} $ and $P_{2}$, if necessary, we may assume that (\ref{5}) holds.
Suppose first that 
\begin{equation}
x_{1}(1-y_{2})-x_{2}(1-y_{1})=0\text{.}  \label{26}
\end{equation}%
By Lemma \ref{L1},\textbf{\ }$x_{2}y_{1}-x_{1}y_{2}\neq 0$ and $%
(1-x_{2})y_{1}-(1-x_{1})y_{2}\neq 0$. Solving (\ref{26}) for $y_{2}$ yields 
\begin{equation}
y_{2}=\dfrac{x_{2}y_{1}+x_{1}-x_{2}}{x_{1}}\text{.}  \label{6}
\end{equation}

Also by (\ref{6}) $x_{1}x_{2}(1-x_{1}-y_{1})(1-x_{2}-y_{2})=\allowbreak
x_{2}^{2}(1-x_{1}-y_{1})^{2}$, which implies that%
\begin{equation}
\sqrt{x_{1}x_{2}(1-x_{1}-y_{1})(1-x_{2}-y_{2})}=x_{2}(1-x_{1}-y_{1})\text{.}
\label{11}
\end{equation}

(\ref{11}) implies that $A_{1}A_{2}=\tfrac{x_{2}}{x_{1}}A_{1}^{2}$, which
implies that $A_{2}=\tfrac{x_{2}}{x_{1}}A_{1}$, and so $\tfrac{A_{1}}{A_{2}}=%
\tfrac{x_{1}}{x_{2}}$. Hence $\tfrac{A_{1}}{A_{2}}=\tfrac{y_{1}}{y_{2}}$ if
and only if $\tfrac{y_{1}}{y_{2}}=\tfrac{x_{1}}{x_{2}}$, which cannot hold
since $x_{2}y_{1}-x_{1}y_{2}\neq 0$, and thus $J\neq 0$ by (\ref{28}). Using
(\ref{6}), (\ref{11}), (\ref{R}), and then simplifying, $R(t)=-\tfrac{\left(
x_{2}-x_{1}\right) ^{2}\left( 1-y_{1}\right) ^{2}}{x_{1}^{2}}\allowbreak
t^{2}+2\allowbreak \tfrac{\left( x_{2}-x_{1}\right) ^{2}\left(
1-y_{1}\right) }{x_{1}}t-\left( x_{2}-x_{1}\right) ^{2}\allowbreak =-\tfrac{%
\left( x_{2}-x_{1}\right) ^{2}{\large (}(1-y_{1})t-x_{1}{\large )}^{2}}{%
x_{1}^{2}}$, which implies that $R\left( \tfrac{x_{1}}{1-y_{1}}\right)
=R^{\prime }\left( \tfrac{x_{1}}{1-y_{1}}\right) =0$; Now let $t=\tfrac{x_{1}%
}{1-y_{1}}$ in the first equation in (\ref{3}): $q_{1}(t)w^{2}+2ty_{1}%
{\large (}(2x_{1}-1)t-x_{1}{\large )}\allowbreak
w+t^{2}y_{1}^{2}=\allowbreak x_{1}^{2}y_{1}\left( 1-w\right) \tfrac{\left(
4x_{1}+3y_{1}-4\right) w+y_{1}}{\left( 1-y_{1}\right) ^{2}}=0$, so $w=\tfrac{%
y_{1}}{4-4x_{1}-3y_{1}}$; Now let $t=\tfrac{x_{1}}{1-y_{1}}$ and $w=\tfrac{%
y_{1}}{4-4x_{1}-3y_{1}}$ \ in the second equation in (\ref{3}): $%
q_{2}(t)w^{2}+2ty_{2}{\large (}(2x_{2}-1)t-x_{2}{\large )}\allowbreak
w+t^{2}y_{2}^{2}=-16\tfrac{(x_{2}-x_{1}){\large (}x_{2}y_{1}+x_{1}-x_{2}%
{\large )}\left( 1-x_{1}-y_{1}\right) ^{2}}{\left( 1-y_{1}\right) ^{2}\left(
3y_{1}-4+4x_{1}\right) ^{2}}=0$ implies that $x_{2}y_{1}+x_{1}-x_{2}=0$. But
by (\ref{6}) one then has $y_{2}=0$. Hence $t=\tfrac{x_{1}}{1-y_{1}}$ yields
no solutions of (\ref{3}). Note that if one lets $t=\tfrac{x_{1}}{1-y_{1}}$
then $w=\tfrac{t}{2}\tfrac{B(t)}{C(t)}=$ $1$. Now $S$ has two distinct
roots, $t_{1}\neq t_{2}$, in $(0,1)$ by Proposition \ref{P4}(i) and (\ref%
{endptsRS}). Note that $t_{1}\neq \tfrac{x_{1}}{1-y_{1}}\neq t_{2}$ by (\ref%
{ReqS}). By Proposition \ref{P5}(i), $C(t_{i})\neq 0$ and if $w_{i}=\tfrac{%
t_{i}}{2}\tfrac{B(t_{i})}{C(t_{i})}$, then $(w_{i},t_{i})$ is a solution of (%
\ref{3}), $i=1,2$. Thus we have two distinct solutions of (\ref{3}). Also,
by Proposition \ref{P5}(ii), there are no other solutions of (\ref{3}) since 
$G$ has precisely the two roots $t_{1}\neq t_{2}$ and the double root at $t=%
\tfrac{x_{1}}{1-y_{1}}$. By Proposition \ref{P2}, there are precisely two
distinct ellipses inscribed in $T$ which pass through $P_{1}$ and $P_{2}$.
Now suppose that 
\begin{equation}
x_{2}y_{1}-x_{1}y_{2}=0\text{.}  \label{32}
\end{equation}

By Lemma \ref{L1},\textbf{\ }$(1-x_{2})y_{1}-(1-x_{1})y_{2}\neq 0$ and $%
x_{1}(1-y_{2})-x_{2}(1-y_{1})\neq 0$. Solving (\ref{32}) for $y_{2}$ yields $%
y_{2}=x_{2}\tfrac{y_{1}}{x_{1}}$, which implies that $J=\allowbreak -\tfrac{%
x_{2}(x_{2}-x_{1})y_{1}^{2}}{x_{1}}\neq 0$. Then $R$ has one root, $t_{1}\in 
$ $(0,1)$, by Proposition \ref{P4}(ii) and (\ref{endptsRS}), and $S$ has one
root, $t_{2}\in $ $(0,1)$, by Proposition \ref{P4}(i) and (\ref{endptsRS}).
Note that $t_{1}\neq t_{2}$ by (\ref{ReqS}). By Proposition \ref{P5}(iii), $%
t_{i}\neq t_{0}$ and $C(t_{i})\neq 0,i=1,2$. By Proposition \ref{P5}(i), $%
C(t_{i})\neq 0$ and if $w_{i}=\tfrac{t_{i}}{2}\tfrac{B(t_{i})}{C(t_{i})}$,
then $(w_{i},t_{i})$ is a solution of (\ref{3}), $i=1,2$. Again we have two
distinct solutions of (\ref{3}). Also, by Proposition \ref{P5}(ii), there
are no other solutions of (\ref{3}) since $G$ has precisely the two roots $%
t_{1}\neq t_{2}$ and since $R(0)=S(0)=0$ implies that $G(0)=G^{\prime }(0)=0$%
. Finally,\textbf{\ }suppose that $(1-x_{2})y_{1}-(1-x_{1})y_{2}=0$. Solving
for $y_{2}$ yields $y_{2}=\tfrac{\left( 1-x_{2}\right) y_{1}}{1-x_{1}}$,
which implies that $J=\allowbreak \tfrac{y_{1}^{2}\left( 1-x_{2}\right)
\left( x_{2}-x_{1}\right) (1-x_{1}-y_{1})}{\left( 1-x_{1}\right) ^{2}}\neq 0$%
. The rest of the proof is very similar to the proof of the case when $%
x_{2}y_{1}-x_{1}y_{2}=0$ and we omit the details.
\end{proof}

\textbf{Example: }$x_{1}=\tfrac{1}{3}$, $y_{1}=\tfrac{1}{5}$, $x_{2}=\tfrac{1%
}{4}$, $y_{2}=\allowbreak \tfrac{2}{5}$; Then $%
x_{1}(1-y_{2})-x_{2}(1-y_{1})=\allowbreak 0,R(t)=\allowbreak -\tfrac{1}{3600}%
\left( 12t-5\right) ^{2}$ \ has a double root at $t=\tfrac{5}{12}$ and

$S(t)=\allowbreak -\tfrac{71}{375}t^{2}+\tfrac{137}{750}t-\tfrac{1}{144}$
has roots $\tfrac{137}{284}\pm \tfrac{77}{426}\sqrt{6}\approx 0.04$ and $%
0.93 $; The two distinct solutions of (\ref{3}) are $\approx (0.04,0.04)$
and $(0.93,0.43)$.

\section{Point and Slope--Interior}

Let $P$ be a point which lies . In this section we answer the following
question: For which points $P$ in the interior of a triangle, $T$, and for
which real numbers $r\in \Re $ is there an ellipse inscribed in $T$ which
passes through $P$ and has slope $r$ at $P$ ? Also, if such an ellipse
exists, how many such ellipses ? Our main result is

\begin{theorem}
\label{OnePoint}Let $P=(x_{0},y_{0})\in \limfunc{int}\left( T\right) =$
interior of the triangle, $T$, with vertices $A,B,C$.

(i) Suppose that $r$ is finite and does not equal the slope of any line
through $P$ and one of the vertices of $T$. Then there is a unique ellipse
inscribed in $T$ which passes thru $P$ and has slope $r$ at $P$.

(ii) There is also a unique ellipse inscribed in $T$ which passes through $P$
and has a vertical tangent line at $P$.

(iii) Suppose that $r$ does equal the slope of some line through $P$ and one
of the vertices of $T$. Then there is no ellipse inscribed in $T$ which
passes through $P$ and has slope $r$ at $P$.
\end{theorem}

As earlier, by affine invariance it suffices to prove the above theorem for
the \textit{unit} triangle, $T$ --the triangle with vertices $(0,0),(1,0)$,
and $(0,1)$. The theorem above then takes the following form:

\begin{theorem}
\label{T2}Let $P=(x_{0},y_{0})$ lie in the interior of the unit triangle.

(i) Let $r\in \Re $ with $y_{0}\neq rx_{0},y_{0}-1\neq rx_{0}$, and $%
y_{0}\neq r(x_{0}-1)$. Then there is a unique ellipse inscribed in $T$ which
passes through $P$ and has slope $r$ at $P$.

(ii) There is also a unique ellipse inscribed in $T$ which passes through $P$
and has a vertical tangent line at $P$.

(iii) Suppose that $y_{0}=rx_{0},y_{0}=1+rx_{0}$, or $y_{0}=r(x_{0}-1)$.
Then there is no ellipse inscribed in $T$ which passes through $P$ and has
slope $r$ at $P$.
\end{theorem}

\begin{remark}
Instead of just looking at the class of ellipses inscribed in triangles, one
might see whether the results of Theorem \ref{OnePoint} still hold for the
class of simple closed convex curves inscribed in a triangle, $T$. It seems
clear geometrically(we have not written out a rigorous proof) that there is
no simple closed convex curve inscribed in $T$\ which passes through $P$ and
has slope $r$ at $P$ if $r$ equals the slope of some line through $P$ and
one of the vertices of $T$. So Theorem \ref{OnePoint}(iii) would still hold
for that case. However, the other parts of Theorem \ref{OnePoint} would not
necessarily hold since they depend on that particular family of simple
closed convex curves.
\end{remark}

\subsection{Preliminary Results}

Since $P=(x_{0},y_{0})\in \limfunc{int}\left( T\right) $,

\begin{gather}
0<x_{0}<1,0<y_{0}<1,  \label{18} \\
x_{0}+y_{0}<1\text{.}  \notag
\end{gather}

We find it convenient to introduce the following notation: 
\begin{eqnarray*}
q_{0}(t) &=&(x_{0}-t)^{2}+4t(1-t)x_{0}y_{0}, \\
h_{0}(t) &=&\left( 1-2x_{0}y_{0}\right) t^{2}+\allowbreak 2x_{0}\left(
y_{0}-1\right) t+x_{0}^{2}\text{.}
\end{eqnarray*}%
We showed in (\ref{q}) that $q_{0}$ has no real roots. Also, the
discriminant of $h_{0}$ equals ${\large (}2x_{0}\left( y_{0}-1\right) 
{\large )}^{2}-4\allowbreak \left( 1-2y_{0}x_{0}\right)
x_{0}^{2}=\allowbreak 4x_{0}^{2}y_{0}\left( 2x_{0}+y_{0}-2\right) <0$. Thus

\begin{equation}
q_{0}(t)\neq 0\neq h_{0}(t),t\in \Re \text{.}  \label{qh0}
\end{equation}%
Differentiating the left hand side of (\ref{2}) with respect to $x$ and
solving for $\dfrac{dy}{dx}$ yields, if $-2xw^{2}t+2xw^{2}+(2x-1)wt-xw+yt%
\neq 0$, 
\begin{equation}
\dfrac{dy}{dx}=\dfrac{w(2ywt^{2}-2yt^{2}+\left( 1-2y\right) wt-xw+yt)}{%
t(-2xw^{2}t+2xw^{2}+(2x-1)wt-xw+yt)}\text{.}  \label{31}
\end{equation}%
Now set $\dfrac{dy}{dx}=r$, let $x=x_{0},y=y_{0}$ in (\ref{2}) and in (\ref%
{31}), and simplify to obtain the following:

\begin{proposition}
\label{P3}Let $P=(x_{0},y_{0})$ lie in the interior of the unit triangle and
let $r\in \Re $. Then there is an ellipse, $E$, which passes through $P$ and
has slope $r$ at $P$ if and only if the following system of equations holds
for some $(w,t)\in S=(0,1)\times (0,1)$:
\end{proposition}

\begin{gather}
q_{0}(t)w^{2}+2ty_{0}{\large (\allowbreak }2tx_{0}-t-x_{0}{\large )}%
\allowbreak w+t^{2}y_{0}^{2}=0  \label{13} \\
{\large (}\left( \allowbreak 2rt^{2}-2rt-1\right) x_{0}+2t(t-1)y_{0}+t%
{\large )}w^{2}  \notag \\
-t{\large (}(2t-1)y_{0}+r(2t-1)x_{0}-rt{\large )}w-r\allowbreak y_{0}t^{2}=0%
\text{.}  \notag
\end{gather}

\subsection{Proof of Theorem \protect\ref{T2}}

\begin{proof}
We shall prove first below that the unique solution of (\ref{13}), with $%
(w,t)\in S$, is given by 
\begin{eqnarray}
w &=&w_{0}=\dfrac{\left( 1-x_{0}-y_{0}\right) \left( rx_{0}-y_{0}\right) ^{2}%
}{\allowbreak q_{w}(r)}  \label{wt2} \\
t &=&t_{0}=\dfrac{\left( 1-x_{0}-y_{0}\right) \left( rx_{0}-y_{0}\right) ^{2}%
}{q_{t}(r)}\text{,}  \notag
\end{eqnarray}%
where 
\begin{eqnarray*}
q_{w}(r) &=&\left( x_{0}^{2}-x_{0}^{3}\right)
r^{2}+2y_{0}x_{0}^{2}r+y_{0}-y_{0}^{2}-x_{0}y_{0}^{2}, \\
q_{t}(r) &=&\left( x_{0}-x_{0}^{2}y_{0}-x_{0}^{2}\right)
r^{2}+2x_{0}y_{0}^{2}r+y_{0}^{2}-y_{0}^{3}\text{.}
\end{eqnarray*}%
Note that $q_{w}$ has discriminant $\left( 2y_{0}x_{0}^{2}\right)
^{2}-4\left( x_{0}^{2}-x_{0}^{3}\right)
(y_{0}-y_{0}^{2}-x_{0}y_{0}^{2})=\allowbreak $

$4x_{0}^{2}y_{0}\left( x_{0}+y_{0}-1\right) <0$ and thus no real roots.
Since $q_{w}$ has positive leading coefficient, we have 
\begin{equation}
\allowbreak q_{w}>0,r\in \Re \text{.}  \label{21}
\end{equation}%
Similarly, $q_{t}$ has discriminant $(2x_{0}y_{0}^{2})^{2}-4\left(
x_{0}-x_{0}^{2}y_{0}-x_{0}^{2}\right) (y_{0}^{2}-y_{0}^{3})=\allowbreak $

$4x_{0}y_{0}^{2}\left( x_{0}+y_{0}-1\right) <0$. Since $%
y_{0}^{2}-y_{0}^{3}=y_{0}^{2}(1-y_{0})>0$, 
\begin{equation}
q_{t}>0,r\in \Re \text{.}  \label{22}
\end{equation}%
We arrived at (\ref{wt2}) by manipulating the equations in (\ref{13}) and
then assuming that certain expressions were not $0$; Some of those
manipulations will appear below, but we prove directly now that $%
(w_{0},t_{0})\in S$ and is a solution of (\ref{13}). By directly
substituting $w=w_{0}$ and $t=t_{0}$ into each equation in (\ref{13}) and
using a computer algebra system, it follows that

$q_{0}(t)w^{2}+2ty_{0}{\large (\allowbreak }2tx_{0}-t-x_{0}{\large )}%
\allowbreak w+t^{2}y_{0}^{2}=\allowbreak 0$ and ${\large (}\left(
\allowbreak 2rt^{2}-2rt-1\right) x_{0}+2t(t-1)y_{0}+t{\large )}w^{2}-t%
{\large (}(2t-1)y_{0}+r(2t-1)x_{0}-rt{\large )}w-r\allowbreak
y_{0}t^{2}=\allowbreak 0$. Now it follows immediately from (\ref{21}) and (%
\ref{22}) that $w_{0},t_{0}>0$. Also, $1-w_{0}=\tfrac{y_{0}\left(
rx_{0}-y_{0}+1\right) ^{2}}{q_{w}(r)}>0$ by (\ref{21}) and $1-t_{0}=\tfrac{%
x_{0}\left( rx_{0}-y_{0}-r\right) ^{2}}{q_{t}(r)}>0$ by (\ref{22}), which
implies that $w_{0},t_{0}<1$ and so $(w_{0},t_{0})\in S$. We prove now that $%
(w_{0},t_{0})$ is the only solution of (\ref{13}) with $(w,t)\in S$. So
suppose that $(w_{1},t_{1})\in S$ is a solution of (\ref{13}). Substitute $%
w=w_{1},t=t_{1}$ into (\ref{13}) and then eliminate the $w_{1}^{2}$ term
from (\ref{13}), multiply thru by $\left( \allowbreak
2rt_{1}^{2}-2rt_{1}-1\right) x_{0}+2t_{1}(t_{1}-1)y_{0}+t_{1}$ in the first
equation in (\ref{13}), multiply thru by $\allowbreak q_{0}(t_{1})$ in the
second equation in (\ref{13}), subtract, factor, and simplify to obtain:%
\begin{equation}
p_{1}(t_{1})w_{1}+y_{0}t_{1}p_{2}(t_{1})=0\text{,}  \label{16}
\end{equation}%
where 
\begin{eqnarray*}
p_{1}(t) &=&\left( -r+2y_{0}-4y_{0}^{2}+2rx_{0}\right) t^{3}+\left(
-3y_{0}-4x_{0}^{2}r+rx_{0}+4y_{0}^{2}\right) \allowbreak t^{2}+ \\
&&x_{0}\left( 2x_{0}^{2}r-2x_{0}y_{0}+rx_{0}+2y_{0}\right) t-x_{0}^{2}\left(
rx_{0}-y_{0}\right) , \\
p_{2}(t) &=&\left( 2y_{0}^{2}-2x_{0}y_{0}r+r\right) t^{2}+\left(
y_{0}+2x_{0}y_{0}r-2rx_{0}-2y_{0}^{2}\right) \allowbreak t+ \\
&&\allowbreak x_{0}\left( rx_{0}-y_{0}\right) \text{.}
\end{eqnarray*}%
Note that we have shown that if (\ref{13}) holds with $w=w_{1},t=t_{1}$,
then (\ref{16}) holds. It is important now to consider the following system
of equations: 
\begin{equation}
p_{1}(t)=0,p_{2}(t)=0\text{.}  \label{15}
\end{equation}

We now prove the following lemma.
\end{proof}

\begin{lemma}
\label{L10} (\ref{15}) holds if and only if $r=r_{0}=\dfrac{y_{0}\allowbreak
\left( 2x_{0}+y_{0}-1\right) }{x_{0}\left( 2x_{0}+y_{0}-2\right) }$ and $t=%
\dfrac{x_{0}}{1-y_{0}}$.
\end{lemma}

\begin{proof}
Note that $2x_{0}+y_{0}-2<x_{0}+y_{0}-1<0$. First, if $r=r_{0}$ and $t=%
\tfrac{x_{0}}{1-y_{0}}$, then direct substitution shows that (\ref{15})
holds. Conversely, suppose that (\ref{15}) holds; Since $p_{1}(x_{0})=%
\allowbreak 4x_{0}^{2}y_{0}^{2}\left( 1-x_{0}\right) \neq 0$, so we may
assume that $t\neq x_{0}$; $p_{1}(t)=0$ implies that $r=y_{0}\tfrac{\left(
4y_{0}-2\right) t^{3}+\left( 3-4y_{0}\right) t^{2}+\left(
2x_{0}^{2}-2x_{0}\right) \allowbreak t-x_{0}^{2}}{\allowbreak \left(
2x_{0}t-t-x_{0}\right) \left( x_{0}-t\right) ^{2}}$and $p_{2}(t)=0$ implies
that $r=y_{0}\tfrac{-2y_{0}t^{2}+\left( 2y_{0}-1\right) t+x_{0}}{\allowbreak
h_{0}(t)}$. If $2x_{0}t-t-x_{0}=0$, then $t=\tfrac{x_{0}}{2x_{0}-1}$, which
lies outside $(0,1)$ if $x_{0}\in (0,1)$, so the denominator of the first
expression for $r$ never vanishes. The denominator of the second expression
for $r$ also never vanishes by (\ref{qh0}). Setting the two expressions for $%
r$ equal to one another yields ${\large (}\left( 4y_{0}-2\right)
t^{3}+\left( 3-4y_{0}\right) t^{2}+\left( 2x_{0}^{2}-2x_{0}\right)
\allowbreak t-x_{0}^{2}{\large )}h_{0}(t){\large -(}\allowbreak \left(
2x_{0}t-t-x_{0}\right) \left( x_{0}-t\right) ^{2}{\large )(}%
-2y_{0}t^{2}+\left( 2y_{0}-1\right) t+x_{0}{\large )}=0$, which implies that 
$2t\left( 1-t\right) q_{0}(t)\left( ty_{0}-t+x_{0}\right) =0$. By (\ref{qh0}%
), $t=\tfrac{x_{0}}{1-y_{0}}$, and substituting back for $t$ in either
expression for $r$ above yields $r=r_{0}$.

We now consider two cases.

\textbf{Case 1:} $r\neq r_{0}$ or $t_{1}\neq \dfrac{x_{0}}{1-y_{0}}$

If $p_{1}(t_{1})=0$, then by Lemma \ref{L10}, $p_{2}(t_{1})\neq 0$, else (%
\ref{15}) would hold with $t=t_{1}$. But then $(w_{1},t_{1})$ cannot be a
solution of (\ref{13}) since (\ref{16}) does not hold. Thus we may assume
that $p_{1}(t_{1})\neq 0$ and (\ref{16}) yields 
\begin{equation}
w_{1}=-\dfrac{y_{0}t_{1}p_{2}(t_{1})}{p_{1}(t_{1})}\text{.}  \label{19}
\end{equation}

Substituting into the first equation in (\ref{13}) using (\ref{19}) yields

$\tfrac{q_{0}(t_{1}){\large (}y_{0}tp_{2}(t_{1}){\large )}^{2}-2ty_{0}%
{\large (\allowbreak }2tx_{0}-t_{1}-x_{0}{\large )(}y_{0}tp_{2}(t_{1})%
{\large )}p_{1}(t_{1})\allowbreak +t_{1}^{2}y_{0}^{2}{\large (}p_{1}(t_{1})%
{\large )}^{2}{\large )}}{{\large (}p_{1}(t_{1}){\large )}^{2}}=0$;
Simplifying and factoring the numerator gives%
\begin{gather}
-4y_{0}^{2}\left( 1-t_{1}\right) t_{1}^{3}q_{0}(t_{1})\left(
t_{1}y_{0}-t_{1}+x_{0}\right) {\large \times }  \label{17} \\
{\large (}-q_{t}(r)t_{1}+\left( 1-x_{0}-y_{0}\right) \left(
rx_{0}-y_{0}\right) ^{2}{\large )}=0\text{.}  \notag
\end{gather}

By (\ref{qh0}), the solution of (\ref{17} ) is $t_{1}=\tfrac{x_{0}}{1-y_{0}}$
or $t_{1}=\tfrac{\left( 1-x_{0}-y_{0}\right) \left( rx_{0}-y_{0}\right) ^{2}%
}{q_{t}(r)}$; If $t_{1}=\tfrac{x_{0}}{1-y_{0}}$, then $\allowbreak $the
first equation in (\ref{13}), with $w=w_{1},t=t_{1}$ becomes $%
x_{0}^{2}y_{0}\left( 1-w_{1}\right) \tfrac{\left( -4+4x_{0}+3y_{0}\right)
w_{1}+y_{0}}{\left( 1-y_{0}\right) ^{2}}=0$, which holds if and only if $%
w_{1}=1\notin S$ or $w_{1}=\tfrac{y_{0}}{4-4x_{0}-3y_{0}}$; Substituting $t=%
\tfrac{x_{0}}{1-y_{0}}$ and $w=\tfrac{y_{0}}{4-4x_{0}-3y_{0}}$ into the
second equation in (\ref{13}) gives $\allowbreak 4x_{0}y_{0}\left(
1-x_{0}-y_{0}\right) \tfrac{\left( 2x_{0}^{2}+y_{0}x_{0}-2x_{0}\right)
r-2y_{0}x_{0}+y_{0}-y_{0}^{2}}{\left( 1-y_{0}\right) ^{2}\left(
-4+4x_{0}+3y_{0}\right) ^{2}}=0\allowbreak $, which holds if and only if $%
r=r_{0}$. Hence $t_{1}\neq \tfrac{x_{0}}{1-y_{0}}$, which implies that $%
t_{1}=\tfrac{\left( 1-x_{0}-y_{0}\right) \left( rx_{0}-y_{0}\right) ^{2}}{%
q_{t}(r)}$, which, when substituted into (\ref{19}) yields $w_{1}=\tfrac{%
\left( 1-x_{0}-y_{0}\right) \left( rx_{0}-y_{0}\right) ^{2}}{q_{w}(r)}$,
which yields the unique solution in \ref{wt2}). That proves that

$(w_{0},t_{0})$ is the only solution in $S$ of (\ref{13}).

\textbf{Case 2: } $r=r_{0}$ and $t_{1}=\dfrac{x_{0}}{1-y_{0}}$

Then $\allowbreak $the first equation in (\ref{13}), with $w=w_{1},t=t_{1}$,
becomes

$x\allowbreak _{0}^{2}y_{0}\left( 1-w_{1}\right) \tfrac{\left(
4x_{0}+3y_{0}-4\right) w_{1}+y_{0}}{\left( 1-y_{0}\right) ^{2}}=0$, and the
second equation in (\ref{13}) becomes $\allowbreak $

$x_{0}y_{0}\left( 2x_{0}+y_{0}-1\right) \left( w_{1}-1\right) \tfrac{\left(
4x_{0}+3y_{0}-4\right) w_{1}+y_{0}}{\left( 2x_{0}+y_{0}-2\right) \left(
1-y_{0}\right) ^{2}}=0$; Note that $%
2x_{0}+y_{0}-2=x_{0}+x_{0}+y_{0}-2<x_{0}-1<0$ by (\ref{18}); Hence $%
(w_{1},t_{1})\in S$ is a solution of (\ref{13}) only if $w_{1}=\tfrac{y_{0}}{%
4-4x_{0}-3y_{0}}$; Letting $r=r_{0}$ in (\ref{wt2}) yields $w_{1}=w_{0}$ and 
$t_{1}=t_{0}$, which proves again that $(w_{0},t_{0})$ is the only solution
in $S$ of (\ref{13}).

To prove Theorem \ref{T2}(ii): Letting $r\rightarrow \pm \infty $ in (\ref%
{wt2}) yields 
\begin{equation}
w_{0}=\allowbreak \dfrac{1-x_{0}-y_{0}}{1-x_{0}},t_{0}=x_{0}\dfrac{%
1-x_{0}-y_{0}}{1-x_{0}(1+y_{0})}\text{.}  \label{33}
\end{equation}%
Substituting the values in (\ref{33}) for $w$ and $t$, respectively, into (%
\ref{31}) and letting $x=x_{0},y=y_{0}$ yields $\tfrac{-2x_{0}y_{0}\left(
1-x_{0}-y_{0}\right) ^{3}}{(1-x_{0})\left( 1-x_{0}-y_{0}x_{0}\right) ^{2}}%
\neq 0$ for the numerator, and $0$ for the denominator.

$\allowbreak $One can then show directly that the solution in (\ref{33}) is
unique, or use the uniqueness of the solution for finite $r$ and use a
perturbation argument. We omit the details. To prove Theorem \ref{T2}(iii),
suppose that $(w_{1},t_{1})\in S$ is a solution of (\ref{13}). If $r=\tfrac{%
y_{0}}{x_{0}}$, then $r=r_{0}$ as well if and only if $\tfrac{2x_{0}+y_{0}-1%
}{2x_{0}+y_{0}-2}=1$, which has no solution.\ Thus $r\neq r_{0}$ and we can
follow the steps above in case 1 to obtain (\ref{17}), which simplifies to $%
-4y_{0}^{2}\left( 1-t_{1}\right) t_{1}^{3}q_{0}(t_{1})\left(
t_{1}y_{0}-t_{1}+x_{0}\right) \tfrac{y_{0}^{2}t_{1}}{x_{0}}=0$ and thus $%
t_{1}=\tfrac{x_{0}}{1-y_{0}}$. Arguing as above in case 1, $t_{1}=\tfrac{%
x_{0}}{1-y_{0}}$ implies that $r=r_{0}$. $\allowbreak $Now suppose that $r=%
\tfrac{y_{0}-1}{x_{0}}$. If $r=r_{0}$ as well, then $y_{0}-1=\tfrac{%
y_{0}\allowbreak \left( 2x_{0}+y_{0}-1\right) }{2x_{0}+y_{0}-2}$, which
implies that $(y_{0}-1)\left( 2x_{0}+y_{0}-2\right) -y_{0}\allowbreak \left(
2x_{0}+y_{0}-1\right) =0$, so that $\allowbreak 2(1-x_{0}-y_{0})=0$, which
cannot hold by (\ref{18}). Thus $r\neq r_{0}$ and (\ref{17}) simplifies to

$4y_{0}^{2}\left( 1-t_{1}\right) t_{1}^{3}q_{0}(t_{1})\left(
t_{1}y_{0}-t_{1}+x_{0}\right) \allowbreak \left( 1-x_{0}-y_{0}\right) \tfrac{%
(1-y_{0})t_{1}+x_{0}}{x_{0}}=0$ and again $t_{1}=\tfrac{x_{0}}{1-y_{0}}$ and
thus $r=r_{0}$. Finally, suppose that $r=\tfrac{y_{0}}{x_{0}-1}$. If $%
r=r_{0} $ as well, then $\left( 2x_{0}+y_{0}-1\right) (x_{0}-1)-x_{0}\left(
2x_{0}+y_{0}-2\right) =0$, so that $1-x_{0}-y_{0}=0$, which cannot hold by (%
\ref{18}). Thus $r\neq r_{0}$ and we can follow the steps above in case 1 to
obtain (\ref{17}), which simplifies to

$-4y_{0}^{2}\left( 1-t_{1}\right) t_{1}^{3}q_{0}(t_{1})\left(
t_{1}y_{0}-t_{1}+x_{0}\right) \allowbreak y_{0}^{2}\left(
1-x_{0}-y_{0}\right) \tfrac{1-t_{1}}{\left( 1-x_{0}\right) ^{2}}$. The rest
follows as for $r=\tfrac{y_{0}}{x_{0}}$ or $r=\tfrac{y_{0}-1}{x_{0}}$. That
proves Theorem \ref{T2}(ii).
\end{proof}

\begin{remark}
There is another somewhat shorter way to prove Theorem \ref{OnePoint}(i) if
one is given the following result proven in \cite{H}.

\textbf{Theorem:} Let $Q$\ be a convex quadrilateral in the $xy$ plane. If $%
(x_{0},y_{0})\in \partial (Q)=$ boundary of $Q$, but $(x_{0},y_{0})$ is not
one of the vertices of $Q$, then there is exactly one ellipse inscribed in $%
Q $\ which passes through $(x_{0},y_{0})$(and is thus tangent to $Q$\ at one
of its sides). Now, given the hypotheses of Theorem \ref{OnePoint}(i), let $%
L $ denote the line thru $P=(x_{0},y_{0})$ with slope $r$. Then one can
prove that $L$, along with the three sides of $T$, form a convex
quadrilateral, $Q$. The theorem above then yields Theorem \ref{OnePoint}(i).
\end{remark}

\textbf{Examples: }(1) $x_{0}=\tfrac{1}{2}$, $y_{0}=\tfrac{1}{4}$, $r=2$.
Using (\ref{wt2}), we have $w=\allowbreak \tfrac{9}{58}$, $t=\allowbreak 
\tfrac{9}{59}$. By (\ref{2}), the corresponding ellipse has equation $%
281\,961x^{2}+272\,484y^{2}-239\,436xy-86\,022\allowbreak x-84\,564y+6561=0$.

$\allowbreak $(2) $x_{0}=\tfrac{1}{3}$, $y_{0}=\tfrac{1}{3}$, $r=\infty $.
Using (\ref{33}), we have $w=\allowbreak \allowbreak \tfrac{1}{2},$ $%
t=\allowbreak \tfrac{1}{5}$. By (\ref{2}), the corresponding ellipse has
equation $25x^{2}+4y^{2}+4xy-10x-4y+1=0$.

\section{Two Points--Boundary}

We now assume that $P_{1}$ and $P_{2}$ lie on the boundary of the triangle, $%
T$.

\begin{theorem}
\label{T3}Let $P_{1}=(x_{1},y_{1})$ and $P_{2}=(x_{2},y_{2})$ be distinct
points which lie on different sides of $\partial \left( T\right) $, and
assume that neither $P_{1}$ nor $P_{2}$ equals one of the vertices of $T$.
Then there is a unique ellipse inscribed in $T$ which is tangent to $%
\partial \left( T\right) $ at $P_{1}$ and at $P_{2}$.
\end{theorem}

\begin{proof}
By affine invariance, it suffices to prove the above theorem for the \textit{%
unit} triangle, $T$, which has sides $L_{1}=\overline{(0,0)\ (1,0)},L_{2}=%
\overline{(1,0)\ (0,1)}$, and $L_{3}=\overline{(0,0)\ (0,1)}$.

\textbf{Case 1:}\textit{\ }$P_{1}\in L_{1}$ and $P_{2}\in L_{3}$. Then $%
P_{1}=(t,0)$ and $P_{2}=(0,w)$ for some $0<t,w<1$. Let $E$ be the ellipse
with equation $F(x,t)=\allowbreak w^{2}x^{2}+t^{2}y^{2}-2wt\left(
2tw-2t+1-2w\right) xy-2tw^{2}x-2t^{2}wy+t^{2}w^{2}=0$. By Proposition \ref%
{P1}(ii), $E$ is tangent to $L_{1}$ at $P_{1}$ and tangent to $L_{3}$ at $%
P_{2}$. That gives existence. To prove uniqueness: Suppose that $\tilde{E}$
is also an ellipse inscribed in $T$ and tangent to $L_{1}$ at $P_{1}$ and
tangent to $L_{3}$ at $P_{2}$. Then by Proposition \ref{P1}(i), $\tilde{E}$
has equation $\tilde{w}^{2}x^{2}+\tilde{t}^{2}y^{2}-2\tilde{w}\tilde{t}%
\left( 2\tilde{t}\tilde{w}-2\tilde{t}+1-2\tilde{w}\right) xy-2\tilde{t}%
\tilde{w}^{2}x-2\tilde{t}^{2}\tilde{w}y+\tilde{t}^{2}\tilde{w}^{2}=0$ for
some $(\tilde{w},\tilde{t})\in (0,1)\times (0,1)$. By Proposition \ref{P1}%
(ii), $\tilde{E}$ is tangent to $L_{1}$ at $(\tilde{t},0)$ and tangent to $%
L_{3}$ at $(0,\tilde{w})$. Since an ellipse cannot be inscribed in $T$ and
tangent to a particular side of $T$ at more than one point(we do not prove
that here), $(\tilde{t},0)=(t,0)$ and $(0,\tilde{w})=(0,w)$, which implies
that $\tilde{t}=t$ and $\tilde{w}=w$ and thus $E=\tilde{E}$ since they have
the same equation.

\textbf{Case 2:}\textit{\ }$P_{1}\in L_{1}$ and $P_{2}\in L_{2}$. Then $%
P_{1}=(t,0)$ and $P_{2}=\left( \tfrac{t(1-w)}{t+(1-2t)w},\tfrac{w(1-t)}{%
t+(1-2t)w}\right) $ for some $0<t,w<1$. Let $E$ be the ellipse with equation 
$F(x,t)=0$. Then by Proposition \ref{P1}(ii), $E$ is tangent to $L_{1}$ at $%
P_{1}$ and tangent to $L_{2}$ at $P_{2}$. That gives existence. To prove
uniqueness: Suppose that $\tilde{E}$ is also an ellipse inscribed in $T$ and
tangent to $L_{1}$ at $P_{1}$ and tangent to $L_{2}$ at $P_{2}$. Arguing as
in case 1, we have $(\tilde{t},0)=(t,0)$ and $\left( \tfrac{\tilde{t}(1-%
\tilde{w})}{\tilde{t}+(1-2\tilde{t})\tilde{w}},\tfrac{\tilde{w}(1-\tilde{t})%
}{\tilde{t}+(1-2\tilde{t})\tilde{w}}\right) =\left( \tfrac{t(1-w)}{t+(1-2t)w}%
,\tfrac{w(1-t)}{t+(1-2t)w}\right) $. Thus $\tilde{t}=t$, $\tfrac{\tilde{t}(1-%
\tilde{w})}{\tilde{t}+(1-2\tilde{t})\tilde{w}}=\tfrac{t(1-w)}{t+(1-2t)w}$,
and $\tfrac{\tilde{w}(1-\tilde{t})}{\tilde{t}+(1-2\tilde{t})\tilde{w}}=%
\tfrac{w(1-t)}{t+(1-2t)w}$. Substituting $\tilde{t}=t$, we have $\tfrac{t(1-%
\tilde{w})}{t+(1-2t)\tilde{w}}=\tfrac{t(1-w)}{t+(1-2t)w}$, which implies
that $\tilde{w}=w$. Again, $E=\tilde{E}$ since they have the same equation.

\textbf{Case 3:}\textit{\ \ }$P_{1}\in L_{2}$ and $P_{2}\in L_{3}$. This is
similar to the other cases and we omit the details.
\end{proof}

\textbf{Example: }$x_{1}=\tfrac{1}{4}$, $y_{1}=\tfrac{3}{4}$, $x_{2}=\tfrac{2%
}{3}$, $y_{2}=0$. One could now solve (\ref{3}), but it is easier to just
use Proposition \ref{P1}(ii). Since we are given two of the three points of
tangency, we have $t=\tfrac{2}{3}$ and $\left( \tfrac{t(1-w)}{t+(1-2t)w},%
\tfrac{w(1-t)}{t+(1-2t)w}\right) =\left( \allowbreak 2\tfrac{1-w}{2-w},%
\tfrac{w}{2-w}\right) =\left( \tfrac{1}{4},\tfrac{3}{4}\right) $, which
implies that $\allowbreak 2\tfrac{1-w}{2-w}=\tfrac{1}{4}$ and $\tfrac{w}{2-w}%
=\tfrac{3}{4}$, each having the solution $w=\tfrac{6}{7}$. By (\ref{2}), the
corresponding ellipse has equation $324x^{2}+196y^{2}+456xy-432x-336y+144=0$.

\end{document}